\documentclass[a4paper,10pt]{amsart}
\usepackage[utf8x]{inputenc}
\usepackage{enumerate}
\usepackage{amssymb}
\usepackage{verbatim}
\usepackage{hyperref}

\newtheorem{theorem}{Theorem}
\newtheorem{proposition}{Proposition}
\newtheorem{lemma}[proposition]{Lemma}
\newtheorem{definition}[proposition]{Definition}
\newtheorem{alemma}{Lemma}

\newcommand{\Norm}{\mathfrak{N}}
\newcommand{\Normk}{\mathfrak{N}_\ks}
\newcommand{\NormK}{\mathfrak{N}_\KL}

\renewcommand{\O}{\mathcal{O}}

\newcommand{\il}[1]{\mathfrak{#1}}
\newcommand{\down}[1]{{}^\mathfrak{d}#1}
\newcommand{\up}[1]{{}^\mathfrak{u}#1}
\newcommand{\Hen}{H_{\mathcal{N}}}
\newcommand{\Qbar}{\overline{\Q}}
\newcommand{\NthX}{N(\theta \ks,X)}
\newcommand{\balf}{{\mbox{\boldmath $\alpha$}}}
\newcommand{\bbalf}{{\mbox{\boldmath $\omega$}}}
\newcommand{\balpha}{\omega}
\newcommand{\s}{\psi}
\newcommand{\iN}{i_\en}
\newcommand{\iNprime}{i_{\en'}}
\newcommand{\vdelta}{{\mbox{\boldmath $\delta$}}}
\newcommand{\normquot}{q}
\newcommand{\gP}{g_P}
\newcommand{\sectwor}{u}
\newcommand\wk{w_\ks}
\newcommand\M{M}
\newcommand\IR{\mathbb R}
\newcommand\IZ{\mathbb Z}

\newcommand\IP{\mathbb P}
\newcommand\IQ{\mathbb Q}
\newcommand{\IS}{\mathbb{S}}
\newcommand{\R}{\mathbb{R}}
\newcommand{\Z}{\mathbb{Z}}
\newcommand{\Q}{\mathbb{Q}}
\newcommand{\IC}{\mathbb{C}}
\renewcommand{\P}{\mathbb{P}}
\newcommand\ks{k}
\newcommand\KL{K}
\newcommand\en{\mathcal{N}}
\newcommand\hen{H_\mathcal{N}}
\newcommand\vnull{{\mbox{\boldmath $0$}}}
\newcommand\A{{\mathfrak{A}}}

\newcommand\D{{\mathfrak{D}}}
\newcommand\B{{\mathfrak{B}}}

\newcommand\F{{\mathfrak{F}}}
\newcommand\vx{{\bf x}}
\newcommand\vy{{\bf y}}
\newcommand\vz{{\bf z}}
\newcommand\vt{{\bf t}}
\newcommand\vNull{{\bf 0}}
\newcommand\Da{m}
\newcommand\TE{t}
\newcommand\Error{\mathcal{E}}
\newcommand\ga{\gamma}
\newcommand\lineqcountingfunction{N_L}
\newcommand{\nb}{{\Gamma}}

\DeclareMathOperator{\Vol}{Vol}
\DeclareMathOperator{\Lip}{Lip}
\DeclareMathOperator{\lcm}{lcm}
\DeclareMathOperator{\Cl}{Cl}
\DeclareMathOperator{\Li}{Li}

\title[Schanuel's theorem for heights defined via extension fields]{Schanuel's theorem for heights defined via extension fields}

\author{Christopher Frei}
\address{Institut für Algebra, Zahlentheorie
  und Diskrete Mathematik, Leibniz Universit\"at Hannover,
  Welfengarten 1, 30167 Hannover, Germany} \email{frei@math.tugraz.at}
\thanks{The first author was supported in part by the FWF project
  \#S9611-N23 and by a Humboldt Research Fellowship for Postdoctoral
  Researchers.}

\author{Martin Widmer}
\address{Department of Mathematics,
Royal Holloway University of London,
Egham TW20 0EX, UK}
\email{martin.widmer@rhul.ac.uk}
\thanks{The second author was supported in part by the FWF grant \#M1222-N13.}

\date{\today}

\dedicatory{}


\begin{document}
\numberwithin{equation}{section}
\numberwithin{proposition}{section}

\begin{abstract}
Let $\ks$ be a number field, let $\theta$ be a nonzero algebraic number, and let $H(\cdot)$ be the Weil height on the algebraic numbers. 
In response to a question by T. Loher and D. W. Masser, we prove an 
asymptotic formula for the number of $\alpha \in \ks$ with $H(\alpha \theta)\leq X$, and 
we analyze the leading constant in our asymptotic formula. In particular, we prove a sharp
upper bound in terms of the classical Schanuel constant.

We also prove an asymptotic counting result for a new class of height functions defined via extension fields of $\ks$ with a fairly 
explicit error term. This provides a conceptual framework for Loher and Masser's problem and generalizations thereof.

Finally, we establish asymptotic counting results for varying $\theta$, namely, for the number of $\sqrt{p}\alpha$ of bounded height, 
where  $\alpha \in \ks$ and $p$ is any rational prime inert in $\ks$.  
\end{abstract}

\maketitle
\tableofcontents

\section{Introduction}
Let $\ks$ be a number field. A well-known result due to Schanuel \cite{25} shows that
the subset of $\ks^n$ of points with absolute multiplicative Weil height no larger than $X$ has cardinality
\begin{alignat*}1
S_\ks(n)X^{d(n+1)}+O(X^{d(n+1)-1}\log X),
\end{alignat*}
as $X$ tends to infinity. Here $d$ is the degree of $\ks$ and the positive constant $S_\ks(n)$ involves all the classical number field invariants; for the definition
see (\ref{Schanuelconstant}).

In the present article we generalize this result in various respects motivated by a question of Loher and Masser. Let $\theta$ be a nonzero algebraic number,
let $H(\cdot)$ denote the absolute multiplicative Weil height on the algebraic numbers $\Qbar$, and write  $\NthX$ 
for the number of $\alpha\in \ks$ with $H(\theta\alpha)\leq X$.

Evertse was the first one to consider the quantity $\NthX$.
The proof of his celebrated uniform upper bounds \cite{Evertse} for the number of solutions of $S$-unit equations over $\ks$ involves the following uniform upper bound 
\begin{alignat*}1
\NthX\leq 5\cdot 2^d X^{3d}+1.
\end{alignat*}
Later Schmidt \cite[Lemma 8B, p.~29]{33} refined Evertse's argument to get the correct exponent on $X$. Schmidt used a different height but elementary 
inequalities between them imply 
\begin{alignat*}1
\NthX\leq 36\cdot 2^{3d}X^{2d}.
\end{alignat*}
But the constant is fairly large. Indeed,
the constant's exponential dependence on $d$ can be removed, as shown by  
Loher and Masser. More precisely, they proved  
\begin{alignat}1\label{LoherMasser}
\NthX\leq 68(d\log d)X^{2d}, 
\end{alignat}
provided $d>1$, and $N(\theta\Q,X)\leq 17 X^2$. (In the special case
$\theta\in \ks$ a similar result was obtained earlier by Loher in his
Ph.D. thesis \cite{30}.) By counting roots of unity they also showed
that an upper bound with a constant of the form $o(d\log \log d)$
cannot hold, and hence regarding the degree their result is nearly
optimal.  Loher and Masser's result (\ref{LoherMasser}) played also an
important role in the recent proof of a longstanding conjecture of
Erd\H{o}s on the largest prime divisor of $2^n-1$ by Stewart
\cite{Stewart2013}. Stewart's strategy builds up on work of Yu
\cite{Yu2007}, \cite{Yu2013} on $p$-adic logarithm forms in which Yu
applies a consequence of (\ref{LoherMasser}) to obtain a significant
improvement. It is this improvement that makes Stewart's approach work
(cf. \cite[p.~378]{Yu2013}).

All the proofs of these upper bounds for $\NthX$ rely in an essential way on the box-principle, which works well for upper
bounds but seems inappropriate to produce asymptotic results. This may have motivated Loher and Masser's following statement \cite[p.~279]{10} regarding their bound 
on $\NthX$:``{\it It would be interesting to know if there are asymptotic formulae like Schanuel's for the cardinalities here, at least for fixed $\theta$ not in $\ks$.}''

Our Theorem \ref{Thmtheta1} responds to this problem for fixed $\theta$ not in $\ks$, and our Theorem \ref{Thmthetan} generalizes Theorem \ref{Thmtheta1} to arbitrary dimensions. 
Theorem \ref{gisatmostone} gives a sharp upper bound for the leading constant in 
these asymptotics in terms of Schanuel's constant $S_\ks(n)$. In Theorem \ref{Thmvartheta}, we shall see asymptotic results for varying $\theta$ not in $\ks$.

To provide a more general framework for Loher and Masser's, and similar questions, we introduce a new class of heights on $\IP^n(\ks)$, 
using finite extensions of the base field $\ks$. As usual, these heights decompose into local factors, one for each place $v$ of $\ks$.   
However, at a finite number of non-Archimedean places, the local factors of these heights do not necessarily arise from norms, 
and moreover, their values do not necessarily 
lie in the value groups of the corresponding places $v$. 
Theorem \ref{generalthm} (in Section \ref{section6}), from which we will deduce Theorem \ref{Thmthetan} (and thus also Theorem \ref{Thmtheta1}), 
is a counting result, in the style of Schanuel's, for 
these heights. 

Our heights are special cases of the heights used by Peyre \cite[D\'efinition 1.2]{Peyre95}. 
Peyre gives asymptotic counting results \cite[Corollaire 6.2.18]{Peyre95}
but no error estimates for his general heights. Therefore the main terms in our Theorem \ref{Thmthetan} and 
Theorem \ref{generalthm} 
could likely be derived from Peyre's result, although with a different representation of the constant. Indeed, a significant part of this
work consists of finding the right representation which enables us to prove the sharp upper bound in Theorem \ref{gisatmostone}, as well as some invariance properties. 
Furthermore, Peyre's approach does not seem to provide comparable error terms, and the latter are essential for the proof of our Theorem
\ref{Thmvartheta}. A very recent result due to Ange \cite[Th\'eor\`eme 1.1]{Ange}  provides a Schanuel type counting result for another special case
of Peyre's heights. Ange also gives a completely explicit and fairly sharp error term. However, his heights require Euclidean/Hermitian norms at the Archimedean places and thus do not include the usual Weil height. 

Next we introduce some notation.
We start with Schanuel's constant $S_\ks(n)$, which is defined as follows
\begin{alignat}3
\label{Schanuelconstant}
S_\ks(n)=\frac{h_\ks R_\ks}{\wk\zeta_\ks(n+1)}
\left(\frac{2^{r}(2\pi)^{s}}{\sqrt{|\Delta_\ks|}}\right)^{n+1}
(n+1)^{r+s-1}.
\end{alignat}
Here $h_\ks$ is the class number, $R_\ks$ the regulator,
$\wk$ the number of roots of unity in $\ks$, $\zeta_\ks$
the Dedekind zeta-function of $\ks$, $\Delta_\ks$ the discriminant,
$r=r_\ks$ is the number of real embeddings of $\ks$ and $s=s_\ks$ is
the number of pairs of complex conjugate embeddings of $\ks$.

For each place $v$ of $\ks$ (or $w$ of $\KL := \ks(\theta)$) we choose the unique absolute value $|\cdot|_v$ on $\ks$ (or $|\cdot|_w$ on $\KL$) that extends either the usual
Euclidean absolute value on $\Q$ or a usual $p$-adic absolute value. 
We also fix a completion $\ks_v$ of $\ks$ at $v$ and for each Archimedean place $v$ of $\ks$ we define a set of points  $(z_0,\ldots,z_n)\in \ks_v^{n+1}$ by 
\begin{alignat*}3
\prod_{w\mid v}\max\{|z_0|_v,|\theta|_w|z_1|_v,\ldots,|\theta|_w|z_n|_v\}^{\frac{[\KL_w:\ks_v]}{[\KL:\ks]}}<1,
\end{alignat*}
where the product runs over all places $w$ of $\KL=\ks(\theta)$ extending $v$. 
As these sets are open, bounded, and not empty, they are measurable and have a finite, positive volume, which we denote by $V_v$. Here we identify $\ks_v$ with $\R$ or with $\IC$, 
and we identify the latter with $\R^2$. We define
\begin{alignat}3\label{defV}
V=V(\theta,\ks,n):=(2^{r}\pi^{s})^{-(n+1)}\prod_{v\mid \infty} V_v.
\end{alignat}

Write $\O_\ks$ for the ring of integers of $\ks$ and let $\mu_\ks$ be the M\"obius function for nonzero ideals of $\O_\ks$. For ideals $A$, $B$ of $\O_\ks$, we write $(A, B) := A + B$. 
Moreover, 
$\Normk A$ denotes the absolute norm of the fractional ideal $A$ of $\ks$. For $\alpha \in \ks$, we also write $\Normk(\alpha) := \Normk(\alpha \O_k)$. Analogous notation 
is used for $\KL$ instead of $\ks$.

For an ideal $B$ of $\O_\ks$, we write $\up B := B \O_\KL$ for the extension of $B$ to $\O_\KL$ (``up''). Similarly, for an ideal $\D$ of $\O_\KL$, we write $\down \D := \D \cap \O_\ks$ 
for the contraction of $\D$ to $\O_\ks$ (``down'').

The dependence on $\theta$ comes in two flavors; while $V$ amounts only to the Archimedean part  
the following constant  captures both parts.

Let $\alpha$ be nonzero and in $\O_\ks$ such that $\alpha\theta \in \O_\KL$, let $\D := \alpha\theta \O_\KL$, and $D := \down{\D}$. We define 
\begin{alignat}3
\label{thetaconstant}
g_\ks(\theta,n) := \frac{V}{\Normk(\alpha)^n}\sum_{B \mid D}\frac{\NormK(\D, \up B)^{\frac{n+1}{[\KL : \ks]}}}
{\Normk B}\sum_{A \mid B^{-1}D}\frac{\mu_\ks(A)}{\Normk A}\prod_{P \mid AB}\frac{\Normk P^{n+1} - \Normk P}{\Normk P^{n+1} - 1}.
\end{alignat}
In the product, $P$ runs over all prime ideals of $\O_\ks$ dividing
$AB$. It will follow from Lemma \ref{lemmaConstantInvariant} that this
definition does not depend on the choice of $\alpha$, and from
Proposition \ref{PropositionProductForm} that $g_\ks(\theta,n)>0$.

\begin{theorem}\label{Thmtheta1}
Let $\theta$ be a nonzero algebraic number, let $\ks$ be a number field and denote its degree by $d$. Then, as $X\geq 1$ tends to infinity, we have 
\begin{alignat*}1
\NthX=g_\ks(\theta,1)S_\ks(1)X^{2d}+O(X^{2d-1}\mathfrak{L}),
\end{alignat*}
where $\mathfrak{L}:=\log(X+1)$ if $d=1$ and $\mathfrak{L}:=1$ otherwise. The implicit constant in the $O$-term depends on $\theta$ and on $\ks$.
\end{theorem}
Let us briefly discuss some properties of the constant $g_\ks(\theta,1)$ and then illustrate the theorem by some examples.

For any nonzero $\alpha$ in $\ks$ we have $\theta \ks=\alpha\theta \ks$. Also, the height is invariant under multiplication by a root of unity.
Therefore $\NthX=N(\zeta\alpha\theta \ks,X)$ for any $\alpha\in \ks^*$ and any root of unity $\zeta$, in particular we have 
\begin{alignat}1\label{invarianceproperty}
g_\ks(\theta,1)=g_\ks(\zeta\alpha\theta,1). 
\end{alignat}
This can also be proved directly from the definition as we shall see in Section \ref{section2}. 
By Schanuel's Theorem we conclude that $g_\ks(\zeta\alpha,1)=1$.
But, as is straightforward to check, the theorem implies even $g_\ks(\zeta\alpha,1)=1$ for $\zeta$ a root of any unit in $\O_\ks$ and $\alpha\in \ks^*$.

The fact that $H(\alpha\theta)=H(\alpha^{-1}\theta^{-1})$ implies 
\begin{alignat*}1
g_\ks(\theta,1)=g_\ks(\theta^{-1},1). 
\end{alignat*}
Next we consider the problem of uniformly bounding $g_\ks(\theta,1)$.
From Schanuel's theorem and the standard inequalities 
$H(\alpha)/H(\theta)\leq H(\theta\alpha)\leq H(\theta)H(\alpha)$ we conclude
\begin{alignat*}1
H(\theta)^{-2d}\leq g_\ks(\theta,1)\leq H(\theta)^{2d}.
\end{alignat*}
This raises the question of the existence of bounds that are uniform in $\theta$ or in $d$,
or even uniform in both quantities $\theta$ and $d$. From (\ref{LoherMasser}) we obtain an upper bound that is uniform in $\theta$, i.e., 
for $d>1$ 
\begin{alignat*}1
g_\ks(\theta,1)\leq \frac{68d\log d}{S_\ks(1)}.
\end{alignat*}
Now if we fix $d>1$ and vary the fields $\ks$ then by the Siegel-Brauer Theorem the right hand-side tends to infinity,
so this bound really depends on $\Delta_\ks$ and not only on $d$. 
However, intuitively one might guess that for most $\alpha\in \ks$ one has $H(\theta\alpha)\geq H(\alpha)$, 
so one might even expect that $g_\ks(\theta,1)\leq 1$ holds true, which, of course, would be best-possible.
We shall answer here all of these questions. We start with the upper bound and confirm the intuitive guess.
\begin{theorem}\label{gisatmostone}
Let $\theta$ be a nonzero algebraic number. Then $g_\ks(\theta, n) \leq 1$. Moreover, equality holds if and only if for every place $v$ of $\ks$ there is an $\alpha_v\in\ks_v$ such that $|\theta|_w = |\alpha_v|_v$ holds for all places $w$ of $\KL$ above $v$.
\end{theorem}

Let us now illustrate Theorem \ref{Thmtheta1} with an example, and thereby explain also the situation regarding lower bounds for $g_\ks(\theta, 1)$. 
Let us first take $\ks=\Q$, and $\theta=\sqrt{p}$ for a prime number $p$.
Then we get the asymptotics
\begin{alignat*}1
\frac{2\sqrt{p}}{p+1}S_{\Q}(1)X^{2}=\frac{24\sqrt{p}}{\pi^2(p+1)}X^2.
\end{alignat*}
More generally, if $p$ is inert in $\ks$ and $\theta=\sqrt{p}$ then we get the asymptotics  
\begin{alignat}1\label{example2}
\frac{2p^{d/2}}{p^d+1}S_{\ks}(1)X^{2d}.
\end{alignat}
Letting $p$ tend to infinity shows that there is no lower bound for $g_\ks(\theta,1)$ that is uniform in $\theta$.
Likewise, fixing a $p$ and taking a sequence $\IQ, \ks_1, \ks_2, \ldots$ of number fields with $p$ inert in $\ks_i$ and $[\ks_i:\IQ]\rightarrow \infty$
shows that there is no lower bound for $g_\ks(\theta,1)$ that is uniform in $d$.

The fast decay of $g_\ks(\sqrt{p},1)$ as $p$ runs over the set $\mathbf{P}_\ks$ (which we define as the set of positive rational primes inert in $\ks$)
suggests another problem. Let
$$\sqrt{\mathbf{P}_\ks}\ks:=\{\sqrt{p}\alpha: p\in \mathbf{P}_\ks, \alpha\in \ks\}=\bigcup_{p \in \mathbf{P}_\ks}\sqrt{p}\ks.$$
The above set has uniformly bounded degree, and thus, by Northcott's Theorem, we may consider its counting function 
$N(\sqrt{\mathbf{P}_\ks}\ks,X):=|\{\beta\in \sqrt{\mathbf{P}_\ks}\ks: H(\beta)\leq X\}|$. Now if $d>2$ then the sum over the terms
in (\ref{example2}) converges, so it is natural to ask whether the asymptotics of $N(\sqrt{\mathbf{P}_\ks}\ks,X)$ are given simply
by summing the asymptotics of $N(\sqrt{p}\ks,X)$ over $\mathbf{P}_\ks$. The following result positively answers this question.

\begin{theorem}\label{Thmvartheta}
Let $\ks$ be a number field of degree $d$. Then, as $X\geq 3$ tends to infinity, we have
\begin{alignat*}1
N(\sqrt{\mathbf{P}_\ks}\ks,X)=
\begin{cases}S_{\ks}(1)X^{4}\log\log X+O(X^{4})&\text{ if } d=2\text,\\
\left(\sum_{\mathbf{P}_\ks}\frac{2p^{d/2}}{p^d+1}\right)S_{\ks}(1)X^{2d}+O(X^{2d-1}\mathcal{L})&\text{ if }d>2\text,
\end{cases}
\end{alignat*}
where $\mathcal{L}=\log\log X$ if $d=3$ and $\mathcal{L}=1$ if $d>3$.  The implicit constant in the $O$-term depends on $\ks$.
\end{theorem}
The case $d=2$ is just slightly more difficult than $d > 2$ and requires additionally Chebotarev's density theorem and
partial summation. However, it is not clear to us how to handle the case $d=1$.

Finally, let us mention that Theorem \ref{Thmtheta1} can also be used to count the elements in the nonzero, e.g., square classes $\ks^*/(\ks^*)^2$. Each class has the form $\gamma\cdot (\ks^*)^2$ with some $\gamma\in\ks^*$.
To count the number $N(\gamma\cdot (\ks^*)^2,X)$ of elements in this square class with height no larger than $X$ we note that $H(\gamma \alpha^2)=H(\sqrt{\gamma}\alpha)^2$, and thus
$N(\gamma\cdot (\ks^*)^2,X)=(1/2)(N(\sqrt{\gamma}\ks,\sqrt{X})-1)$. E.g., the square class $(\Q^*)^2$ has asymptotically $(6/\pi^2)X$ elements whereas the square class
$11\cdot(\Q^*)^2$ has asymptotically only $(\sqrt{11}/\pi^2)X$ elements of height bounded by $X$.\\

Next we generalize Theorem \ref{Thmtheta1} to higher dimensions.  Let
$N(\theta \ks^n,X)$ be the number of points $\balf=(\alpha_1,\ldots,
\alpha_n)\in \ks^n$ with $H((\theta\alpha_1,\ldots,
\theta\alpha_n))\leq X$.  Of course, here $H:\Qbar^n\rightarrow [1,\infty)$ is the (affine) absolute
multiplicative Weil height, defined by
\begin{equation*}
  H(\omega_1, \ldots, \omega_n)^{[\KL:\Q]} := \prod_{w\in M_\KL}\max\{1,|\omega_1|_w,\ldots, |\omega_n|_w\}^{d_w},
\end{equation*}
where $K$ is any number field containing $\omega_1, \ldots, \omega_n$,
the index $w$ runs over the set $M_\KL$ of all places of $\KL$, and
$d_w := [\KL_w : \Q_w]$ denotes the local degree, where $\Q_w$ is the completion of $\Q$
with respect to the place below $w$.

\begin{theorem}\label{Thmthetan}
Let $\theta$ be a nonzero algebraic number, let $\ks$ be a number field, denote its degree by $d$, and let $n$ be a positive rational integer. 
Then, as $X\geq 1$ tends to infinity, we have 
\begin{alignat*}1
N(\theta \ks^n,X)=g_\ks(\theta,n)S_\ks(n)X^{d(n+1)}+O(X^{d(n+1)-1}\mathfrak{L}),
\end{alignat*}
where $\mathfrak{L}:=\log(X+1)$ if $(n,d)=(1,1)$, and $\mathfrak{L}:=1$ otherwise. The implicit constant in the $O$-term depends 
on $\theta$, on $\ks$, and on $n$.
\end{theorem}
Of course the invariance property (\ref{invarianceproperty}) remains valid for arbitrary $n$ instead of $1$.
Ange \cite[Corollaire 1.6]{Ange} has shown a related result (although with different choice of the height); instead of fixing one $\theta$ he allows a different $\theta$ for each coordinate and his
error term is completely explicit and quite sharp. On the other hand he requires that a (positive) power of each $\theta$ lies in the ground field $\ks$.

So far we have counted elements $\theta\balf$ in $\theta\ks^n$ of bounded height.
What if we replace the set $\theta \ks$ by $\theta+\ks$? Or $\theta \ks^2$ by $\theta_1 \ks\times \theta_2 \ks$? 
More generally, we suppose $L_1,\ldots, L_n$ are linearly independent linear forms in $n$ variables with coefficients in $\Qbar$ and $\theta_1,\ldots, \theta_n$ are in $\Qbar$.
Suppose we want to count elements of bounded height in the set 
\begin{alignat*}1
\{(L_1(\balf)+\theta_1,\ldots,L_n(\balf)+\theta_n): \balf\in \ks^n\}.
\end{alignat*}
Now let $\balf:=(\balpha_1/\balpha_0,\ldots,\balpha_n/\balpha_0)\in \ks^n$ and define $\bbalf:=(\balpha_0,\ldots,\balpha_n)$. Then
\begin{alignat*}1
H((L_1(\balf)+\theta_1,\ldots,L_n(\balf)+\theta_n))=\prod_{w}\max\{|\mathcal{L}_0(\bbalf)|_w,\ldots, |\mathcal{L}_n(\bbalf)|_w\}^{\frac{[\KL_w : \Q_w]}{[\KL : \Q]}},
\end{alignat*}
where $\mathcal{L}_0(\bbalf)=\balpha_0$ and $\mathcal{L}_i(\bbalf)=L_i(\balpha_1,\ldots,\balpha_n)+\theta_i\balpha_0$ (for $1\leq i\leq n$), which give us $n+1$ linearly independent
linear forms. Here the right hand-side defines a special case of a so-called adelic Lipschitz height $\Hen$ (introduced in \cite{art1}) on $\P^n(\KL)$, 
where $\KL$ is any number field containing $\ks$, and
the coefficients of $\mathcal{L}_0,\ldots,\mathcal{L}_n$, and the product runs over all places $w$ of $\KL$. Thus, we need to count the points 
$P=(\balpha_0:\cdots:\balpha_n)\in \P^n(\ks)$ with
$\balpha_0\neq 0$ and $\Hen(P)\leq X$.

These generalizations of Loher and Masser's problem naturally motivate our general theorem (Theorem \ref{generalthm}), which is as follows. 
Given two number fields $\ks\subseteq \KL$ and an adelic Lipschitz height $\Hen$ on $\KL$, we give an asymptotic formula for the number of points $P\in \mathbb P^n(\ks)$ with 
$\Hen(P)\leq X$,
as the parameter $X$ tends to infinity. To be more accurate, we also impose a minor additional assumption on the adelic Lipschitz height $\hen$,
which seems fulfilled in all natural applications, in particular, it holds in the aforementioned examples.

The special case $\KL=\ks$ of our general theorem follows from a result in \cite{art1}. There, a complementary 
result was proved, in the sense that points of $\P^n(\KL)$ defined over a proper subextension of $\KL/\ks$ were excluded from the counting
(which is insignificant for the main term but was needed to obtain good error terms).

Now already with general linear forms as above it seems unlikely that the main term can be brought into an as 
civilized form as for Theorem \ref{Thmthetan} (see also the remark in \cite[p. 1766 third paragraph]{art3}). Indeed, a considerable part of our work 
consists of finding the simple representation of the constant in the special case of Theorem \ref{Thmthetan}.
However, it turns out that the given representation is not so convenient for theoretical considerations. Indeed, even the most
obvious properties, such as the invariance property (\ref{invarianceproperty}), are not immediately clear from the present definition.
In Section \ref{section2} we establish a representation of $g_\ks(\theta,n)$ as a product of local factors (Proposition \ref{PropositionProductForm}), which 
is a first step in the proof of Theorem \ref{gisatmostone} and also reveals the invariance property (\ref{invarianceproperty}). 
  
At any rate, a situation involving linear forms similar to the above turns up if we want to count solutions of a system of linear equations with certain restrictions 
to the coordinates of the solutions.
Here is an example. Consider the equation 
\begin{alignat}1\label{lineq}
\sqrt{2}x+\sqrt{3}y+\sqrt{5}z=0,
\end{alignat}
defined over $\KL=\Q(\sqrt{2},\sqrt{3},\sqrt{5})$. Using arguments from \cite{art3} one can easily compute that the number of solutions $(x,y,z)\in \KL^3$
with $H((x,y,z))\leq X$ is asymptotically given by
\begin{alignat*}1
\left(\frac{\sqrt{96}-(\sqrt{2}+\sqrt{3}-\sqrt{5})^2}{\sqrt{480}}\right)^8S_K(2)X^{24}+O(X^{23}).
\end{alignat*}
But what about the number of such solutions whose first two coordinates are rational?
This question reduces to counting the elements $(\balpha_0:\balpha_1:\balpha_2)\in \P^2(\Q)$ with bounded adelic Lipschitz height
\begin{alignat*}1
\Hen((\balpha_0:\balpha_1:\balpha_2)) =
\prod_{w}\max\{|\balpha_0|_w,|\balpha_1|_w,|\balpha_2|_w,|\frac{\sqrt{2}\balpha_1+\sqrt{3}\balpha_2}{\sqrt{5}}|_w\}^{\frac{[\KL_w : \Q_w]}{[\KL : \Q]}}.
\end{alignat*}
Applying our general theorem gives the following asymptotic formula
\begin{alignat}1\label{lineqexample}
\lineqcountingfunction(X) = V_{\en'} \cdot \frac{1}{62\zeta(3)}\cdot(1 + 2\cdot 5^{1/4} + 4 \cdot 5^{-1/2}) X^3 + O(X^2) 
\end{alignat}
for the number $\lineqcountingfunction(X)$ of solutions $(x,y,z)$ of (\ref{lineq}) of height bounded by $X$ and with $x,y\in \Q$.
Here $V_{\en'}$ denotes the volume of the set of points $(z_0,z_1,z_2)$ in $\IR^3$ that satisfy the inequality 
$$\max\{|z_0|,|z_1|,|z_2|,|\sqrt{2}z_1+\sqrt{3}z_2|/\sqrt{5}\}\max\{|z_0|,|z_1|,|z_2|,|\sqrt{2}z_1-\sqrt{3}z_2|/\sqrt{5}\}<1.$$
For the computations we refer the reader to the appendix. 
 
Finally, by Northcott's theorem there is no need to restrict to a fixed number field, and one could also consider all number fields of a given fixed degree
simultaneously. 
Let us define the set 
\begin{alignat*}1
\theta k(n;e)=\{(\theta\alpha_1,\ldots,\theta\alpha_n): [k(\alpha_1,\ldots,\alpha_n):k]=e\}.
\end{alignat*}
So Theorem \ref{Thmthetan} gives the asymptotics for the counting function $N(\theta k(n;1),X)=N(\theta k^n,X)$, 
and more generally, one could ask for the asymptotics of $N(\theta k(n;e),X)$.
The special case $\theta \in k$ was considered in  \cite{22}, \cite{14}, \cite{7}, \cite{37}, \cite{1}, and \cite{art2}. Indeed, it is likely that the
methods from \cite{art1} and \cite{art2}, combined with those of the present article, are sufficient to solve this problem, provided $n$ is large enough.
On the other hand, it would be interesting to know whether Masser and Vaaler's approach from \cite{1} can be combined with ours to handle the case $n=1$. \\

The plan of the paper is as follows. In Section \ref{section2} we establish a product representation of $g_\ks(\theta,n)$, and we use this to
deduce some of its properties. This product form is also the starting point in the proof of Theorem \ref{gisatmostone}, which we give in Section \ref{section3}.
Then in Section \ref{section4} we state and prove some basic facts about lattice points, which are required for the proofs of Theorem \ref{generalthm} and Theorem \ref{Thmvartheta}.
Section \ref{section5} provides the necessary notions such as adelic Lipschitz systems to state our general theorem. Then in Section \ref{section6} we state the general
theorem (Theorem \ref{generalthm}), and in Section \ref{section7} we give its proof. From Theorem \ref{generalthm} we deduce Theorem \ref{Thmthetan}, which is done in 
Section \ref{section8}.
The proof of Theorem \ref{Thmvartheta} is carried out in Section \ref{section9}. Finally, in the appendix we calculate formula (\ref{lineqexample}) using
Theorem \ref{generalthm}. \\

By a prime ideal we always mean a nonzero prime ideal. By $E \unlhd \O_\ks$, we mean that $E$ is a nonzero ideal of $\O_\ks$. An empty product is always interpreted as $1$, and an empty sum is interpreted as $0$. 

\section{Product representation and invariance properties of the constant}\label{section2}
In this section, we use a product representation for the constant $g_\ks(\theta,n)$ to derive some of its properties. Let $\D$, $B$ be nonzero ideals of $\O_\KL$ or $\O_\ks$,
respectively. For convenience, we define
\[\normquot(\D, B) := \normquot(\D, B, n) := \frac{\NormK(\D, \up{B})^{(n+1)/[\KL:\ks]}}{\Normk B}\text.\]
Clearly, $\normquot(\D, B)$ is multiplicative in $B$, by which we mean that $\normquot(\D,B_1B_2) = \normquot(\D,B_1)\normquot(\D,B_2)$ whenever $(B_1, B_2) = 1$. Moreover, $\normquot(\D, B) = \normquot((\D, \up B), B)$, and if $B_1 \mid B_2$, then $\normquot(\up B_2\D, B_1) = \Normk B_1^n$ and $\normquot(\up B_1\D, B_2) = \Normk B_1^n \normquot(\D, B_1^{-1}B_2)$.

We now define local factors at prime ideals $P$ of $\O_\ks$, by
\begin{align*}
  \gP(\D, n) := \frac{\Normk P - 1}{\Normk P^{n+1}-1}\left(1+(\Normk P^n - 1)\sum_{j=0}^\infty \normquot(\D,P^j)\right).
\end{align*}
Let $v_P$ denote the $P$-adic valuation on $\ks$, normalized by $v_P (\ks^*) = \IZ$.
The infinite sum converges, since
\begin{equation}\label{eq:normquot_relation}
\normquot(\D,P^j) = \Normk
P^{v_P(\down \D)-j}\normquot(\D,P^{v_P(\down \D)})
\end{equation}
holds for all $j
\geq v_P(\down \D)$. 
Clearly, $\gP(\D, n) = \gP(\D_P, n)$, where $\D_P
:= \prod_{\mathfrak{P}\mid P}\mathfrak{P}^{v_\mathfrak{P}(\D)}$ is the
part of $\D$ lying over $P$.
\begin{lemma}\label{lemmaProductForm}
Let $\D$ be a nonzero ideal of $\O_\KL$ and $D := \down\D$. Then 
\[\sum_{B \mid D}\normquot(\D, B)\sum_{A \mid B^{-1}D}\frac{\mu_\ks(A)}{\Normk A}\prod_{P \mid AB}\frac{\Normk P^{n+1} - \Normk P}{\Normk P^{n+1} - 1} = \prod_{P}\gP(\D, n)\text.\]
\end{lemma}
\begin{proof}
We start by investigating the expression
\[S(D, B) := \sum_{A \mid B^{-1}D}\frac{\mu_\ks(A)}{\Normk A}\prod_{P \mid AB}\frac{\Normk P^{n+1} - \Normk P}{\Normk P^{n+1} - 1}\]
for a given ideal $B$ of $\O_\ks$ dividing $D$. Clearly,
\[S(D, B)=\prod_{P \mid B}\frac{\Normk P^{n+1} - \Normk P}{\Normk P^{n+1} - 1}\sum_{A \mid B^{-1}D}f(A)\text,\]
where
\[f(A) := \frac{\mu_\ks(A)}{\Normk A}\prod_{\substack{P \mid A\\ P \nmid B}}\frac{\Normk P^{n+1} - \Normk P}{\Normk P^{n+1} - 1}\text.\]
The function $f$ is multiplicative and $f(\O_\ks) = 1$. For any prime ideal $P$ dividing $B^{-1}D$, we have
\[f(P) = \begin{cases}-\Normk P^{-1}&\text{ if } P \mid B\text,\\-(\Normk P^n - 1)/(\Normk P^{n+1} - 1)&\text{ if }P \nmid B\text.\end{cases}\]
Moreover, $f(P^e) = 0$ if $e > 1$. We use 
\begin{equation*}
 \sum_{A \mid B^{-1}D}f(A) = \prod_{P \mid B^{-1}D}(1 + f(P))
\end{equation*}
to obtain
\begin{align*}
S(D, B) &= \prod_{P \mid B}\frac{\Normk P^{n+1} - \Normk P}{\Normk P^{n+1} - 1}\prod_{\substack{P \mid B^{-1}D\\P\nmid B}}\frac{\Norm P^{n+1} - 
\Normk P^n}{\Normk P^{n+1} - 1}\prod_{P \mid (B^{-1} D, B)}\frac{\Norm_k P - 1}{\Normk P}\\
&= \prod_{P \mid D}\frac{\Normk P^{n+1} - \Normk P^n}{\Normk P^{n+1} - 1}\prod_{P \mid B}\frac{\Normk P^{n+1} - \Normk P}{\Normk P^{n+1} - 
\Normk P^n}\prod_{P \mid (B^{-1} D, B)}\frac{\Norm_k P - 1}{\Normk P}\text.
\end{align*}
Let $T(D, B) := S(D, B)/\prod_{P \mid D}\frac{\Normk P^{n+1} - \Normk P^n}{\Normk P^{n+1} - 1}$. Then the expression on the left-hand side of the Lemma is given by
\[\left(\prod_{P \mid D}\frac{\Normk P^{n+1} - \Normk P^n}{\Normk P^{n+1} - 1}\right) \sum_{B \mid D} \normquot(\D, B)T(D, B)\text.\]
Since both $T(D, B)$ and $\normquot(\D, B)$ are multiplicative in $B$, this is equal to
\begin{align}\label{eq:product_form_1}
& \prod_{P \mid D}\left(\frac{\Normk P^{n+1} - \Normk P^n}{\Normk P^{n+1} - 1}\sum_{j=0}^{v_P(D)}\normquot(\D, P^j)T(D, P^j)\right)\text.
\end{align}
Elementary manipulations show that
\[T(D, P^j) = \frac{(\Normk P^n - 1)(\Normk P - 1)}{\Normk P^{n+1} - \Normk P^n}\cdot \begin{cases}\displaystyle
\frac{\Normk P^{n+1} - \Normk P^n}{(\Normk P^n - 1)(\Normk P - 1)} &\text{ if }j=0,\\\\
1 &\text{ if } 1 \leq j < v_P(D)\text,\\
\ &\ \\
\normalsize\displaystyle\sum_{j=v_P(D)}^\infty\Normk P^{v_P(D)-j}&\text{ if } j = v_P(D)\text.\end{cases}\]
Using \eqref{eq:normquot_relation}, this shows that each of the factors in \eqref{eq:product_form_1} has the form
\begin{equation*}
  \frac{(\Normk P-1)(\Normk P^n-1)}{\Normk P^{n+1}-1}\left(\frac{\Normk P^{n+1}-\Normk P^n}{(\Normk P^n - 1)(\Normk P - 1)} + \sum_{j=1}^\infty \normquot(\D,P^j)\right) = g_P(\D,n).
\end{equation*}
\end{proof}

Lemma \ref{lemmaProductForm} with $\D := \alpha \theta \O_\KL$ yields the following formula for $g_\ks(\theta, n)$. 
\begin{proposition}\label{PropositionProductForm}
If $\alpha$ is nonzero and in $\O_\ks$ with $\alpha \theta \in \O_\KL$ then
\begin{equation}\label{constantProductForm}
g_\ks(\theta, n) = \frac{V}{\Normk(\alpha)^n}\prod_P \gP(\alpha\theta\O_\KL, n)\text. 
\end{equation}
\end{proposition}
The next lemma shows that $g_\ks(\theta, n)$ does not depend on the choice of $\alpha$.
\begin{lemma}\label{lemmaConstantInvariant}
Let $A$ be a nonzero ideal of $\O_\ks$ and $\D$ a nonzero ideal of $\O_K$. Then 
\[\gP(\up A \D, n) = \Normk P^{nv_P(A)}\gP(\D, n)\text.\] 
\end{lemma}

\begin{proof}
We have
\[\normquot(\up A \D, P^j) =\begin{cases}\Normk P^{nj}&\text{ if }0 \leq j < v_P(A)\text{, }\\
\Normk P^{n v_P(A)}\normquot(\D, P^{j-v_P(A)})&\text{ if }j \geq v_P(A).\end{cases}\]
The lemma follows by inserting these expressions for $\normquot(\up A\D,P^j)$ in the definition of $\gP(\up A\D,n)$.
\end{proof}

Given nonzero $\alpha$, $\beta \in \O_\ks$ such that $\alpha \theta$, $\beta \theta \in \O_\KL$, then we have 
\[\Normk(\alpha)^n\prod_P \gP(\beta\theta\O_\KL, n) = \prod_P \gP(\alpha\beta\theta\O_\KL, n) = \Normk(\beta)^n\prod_P \gP(\alpha\theta\O_\KL, n)\text,\]
which shows the independence of $g_\ks(\theta, n)$ from the choice of $\alpha$.

To see invariance property \eqref{invarianceproperty} directly from \eqref{constantProductForm}, we need the following lemma. 

\begin{lemma}\label{lemmaVolumeInvariant}
Let $\alpha \in \ks^*$. Then
\[V(\alpha \theta, \ks, n) = \frac{V(\theta, \ks, n)}{\Normk(\alpha)^n}\text.\]
\end{lemma}

\begin{proof}
For any Archimedean place $v$ of $\ks$, the map $\phi_v : \ks_v^{n+1} \to \ks_v^{n+1}$ defined by $\phi_v(z_0, z_1 \ldots, z_n) = (z_0, |\alpha|_v z_1, \ldots, |\alpha|_v z_n)$ is a linear automorphism of $\ks_v^{n+1}$ (considered as $\R^{[\ks_v : \R](n+1)})$ of determinant $|\alpha|_v^{[\ks_v : \R] n}$. Therefore, $|\alpha|_v^{[\ks_v : \R] n}V_v(\alpha \theta, \ks, n) = V_v(\theta, \ks, n)$.
\end{proof}

Lemma \ref{lemmaConstantInvariant} and Lemma \ref{lemmaVolumeInvariant} imply that 
\begin{equation}\label{constantInvariant}
g_\ks(\alpha \theta, n) = g_\ks(\theta, n) 
\end{equation}
for every nonzero $\alpha \in \O_\ks$, and hence for every $\alpha \in \ks^*$. In particular, it suffices to prove Theorem \ref{gisatmostone} and 
Theorem \ref{Thmthetan} for integral $\theta$.

\section{Proof of Theorem \ref{gisatmostone}}\label{section3}
We start off by estimating the volume $V(\theta, \ks, n)$.
\begin{lemma}\label{lemmaVolumeEstimate}
We have
\[V(\theta, \ks, n) \leq \NormK(\theta)^{-n/[\KL : \ks]}\text.\]
Moreover, equality holds if and only if for every Archimedean place $v$ of $\ks$ the absolute values 
$|\theta|_w$ are equal for all $w\mid v$.
\end{lemma}

\begin{proof}
  Let $v$ be an Archimedean place of $\KL$, and let $p_v=p_v(\theta)
  := \prod_{w \mid v}|\theta|_w^{\frac{[\KL_w : \ks_v]}{[\KL :
      \ks]}}$. Consider the functions $f_v^{(1)},f_v^{(2)}:\ks_v^{n+1} \to \R$
  given by
  \begin{align*}
    f_v^{(1)}(z_0, \ldots, z_n) &:= \prod_{w \mid v}\max\{|z_0|_v, |\theta|_w |z_1|_v, \ldots, |\theta|_w |z_n|_v\}^{\frac{[\KL_w : \ks_v]}{[\KL : \ks]}},\\
    f_v^{(2)}(z_0, \ldots, z_n) &:= \max\left\{|z_0|_v\text{, }p_v|z_1|_v\text{, }\ldots\text{, }p_v|z_n|_v\right\}.
  \end{align*}
  Then $f_v^{(i)}(t \mathbf{z}) = |t|_vf_v^{(i)}(\mathbf{z})$ holds
  for all $t \in \ks_v$, $\mathbf{z}\in \ks_v^{n+1}$, and
  $i\in\{1,2\}$.  Moreover, $f_v^{(1)} \geq f_v^{(2)}$ as functions on
  $\ks_v^{n+1}$, with equality if and only if $|\theta|_w$ is constant
  on $w\mid v$.

  Now $\Vol\{\mathbf{z}\in \ks_v^{n+1}: f_v^{(1)}(\mathbf{z})<1\}
  \leq \Vol\{\mathbf{z}\in \ks_v^{n+1}: f_v^{(2)}(\mathbf{z})< 1\}$,
  with equality if and only if $f_v^{(1)} = f_v^{(2)}$. Evaluating
  both volumes gives
\begin{alignat}1\label{Vvineq}
V_v \leq p_v^{-n [\ks_v : \R]}\cdot 
\begin{cases}2^{n+1}&\text{if $v$ is real,}\\\pi^{n+1}&\text{if $v$ is complex,}\end{cases}
\end{alignat}
with equality if and only if $|\theta|_w$ is constant on $w\mid v$. Thus,
\[V(\theta, \ks, n) \leq \prod_{w \mid \infty} |\theta|_w^{-\frac{n
    [\KL_w : \R]}{[\KL : \ks]}} = \NormK(\theta)^{-n/[\KL :
  \ks]}\text,\] with equality if and only if the condition in the
lemma is satisfied.
\end{proof}

Let us recall some simple facts, which will be used in the sequel 
without further notice. Let $A$, $B$ be ideals of $\O_\ks$, and let $\il A$, $\il B$ be ideals of $\O_\KL$. Moreover, suppose that $P$ is a prime ideal of $\O_\ks$ and that $\il P$ runs over all prime ideals of $\O_\KL$ above $P$. Then
\begin{itemize}
 \item $v_P(\down \il A) = \max_{\il P \mid P}\{\lceil v_{\il P}(\il A)/e_{\il P} \rceil\}$
 \item $\down\up A = A$
 \item $\il A \mid \up \down \il A$
 \item $\up(A B) = \up A \up B$
 \item $\il A \mid \up A$ if and only if $\down{\il A} \mid A$
\end{itemize}

\begin{lemma}\label{lemmaConstantEstimate}
Let $\D$ be a nonzero ideal of $\O_\KL$ and $P$ a prime ideal of $\O_\ks$. Then 
\[\gP(\D, n) \leq \NormK(\D_P)^{n/[\KL : \ks]}\text,\]
with equality if and only if $\D_P = \up \down \D_P$.
\end{lemma}

\begin{proof}
  Lemma \ref{lemmaConstantInvariant} and the fact that $\gP(\D, n) =
  \gP(\D_P, n)$ imply equality if $\D_P = \up\down \D_P$. Therefore,
  let us assume that $v_P(\down \D)=:l\geq 1$ and that $\D_P$ is a
  proper divisor of $\up \down \D_P = \up P^l$. Again by Lemma
  \ref{lemmaConstantInvariant}, we may assume that $\up P \nmid \D$.

Let
\[\sectwor:= \frac{1}{[\KL : \ks]}\sum_{\mathfrak P \mid P}f_{\mathfrak P}v_\mathfrak{P}(\D)\text,\]
where the sum runs over all prime ideals $\mathfrak P$ of $\O_\KL$
lying over $P$, and $f_{\mathfrak P} = f_{\mathfrak P\mid P} =
[\O_\KL/\mathfrak P : \O_\ks/P]$ is the relative degree of $\mathfrak
P$ over $P$. Then the right-hand side in the lemma is just
$\Normk(P)^{n \sectwor}$. Since $v_P(\down \il D) = l \geq 1$, we get
$\sectwor > 0$. Let $e_\il{P} = e_{\mathfrak{P}\mid P}$ be the
ramification index of $\mathfrak{P}$ over $P$. As $\il D_P$ is a
proper divisor of $\up\down\il D_P$, we conclude that $v_{\il
  P}(\il D)<e_{\il P} v_P(\down \il D)$ for at least one
$\il{P}\mid P$. Thus,
\begin{equation*}
\sum_{\il P \mid P}f_{\il P} v_{\il P}(\il D) < \sum_{\il P\mid P}f_{\il P} e_{\il P}v_P(\down \il D) = [\KL : \ks]\cdot l\text,
\end{equation*}
and therefore $\sectwor < l$. Similarly, we have
\[\normquot(\D, P^j) = \Normk(P)^{\frac{n+1}{[\KL : \ks]}\left(\sum_{\mathfrak P \mid P}f_{\mathfrak P}\min\{v_{\mathfrak P}(\D), j e_{\mathfrak P}\}\right) - j}\text,\]
for any $j \geq 0$. By our assumption that $\up P \nmid \il D$, we have $v_{\il P}(\il D) < j e_{\mathfrak P}$ for some $\il P \mid P$ and all $j \geq 1$.  Replacing all the minima in the above formula by their second arguments yields
\begin{equation}\label{q_first_estimate}
\normquot(\D, P^j) < \Normk P^{j n}\text.
\end{equation}
Similarly, replacing the minima by their first arguments yields
\begin{equation}\label{q_second_estimate}
\normquot(\D, P^j) \leq \Normk(P)^{(n+1)\sectwor - j},
\end{equation}
and the inequality is strict if and only if $j < l$. Let $1 \leq L
\leq l$ be the integer with $L-1 < \sectwor \leq L$. We use
\eqref{q_first_estimate} for $j < L$ and \eqref{q_second_estimate} for
$j \geq L$ to estimate $\normquot(\D,P^j)$ in the definition of
$g_P(\D,n)$. This shows that $g_P(\D,n)$ is bounded from above by
\begin{equation*}
  \frac{1}{\Normk P^{n+1}-1}\left(\Normk P^{Ln+1} - \Normk P^{Ln} + \Normk P^{(n+1)\sectwor - L + n+1} - \Normk P^{(n+1)\sectwor-L+1}\right)\text,
\end{equation*}
with a strict inequality whenever $l>1$. To prove the lemma, it is
enough to show that this expression is bounded by $\Normk P^{n u}$
(with strict inequality if $l=1$). To this end, let $h$ be the
function given by
\begin{align*}
h(x) :&= x^{n\sectwor+n+1} - x^{n\sectwor} + x^{(n+1)\sectwor - L+1} - x^{(n+1)\sectwor - L + n + 1} + x^{Ln} - x^{Ln+1}\text.
\end{align*}
Hence, we need to show that $h(\Normk P) \geq 0$, with a strict inequality if $l = 1$. 
With $\tilde{\sectwor} := \sectwor - L +1\in (0, 1]$ and
\[h_1(x) := x^{n\tilde{\sectwor} + n + 1} - x^{n\tilde{u}} + x^{(n+1)\tilde \sectwor} - x^{(n+1)\tilde{\sectwor} + n} + x^n - x^{n+1}\text,\]
we have $h(x) = x^{n(L-1)} h_1(x)$. If $\tilde{\sectwor} = 1$ then $h_1(x) \equiv 0$. We observe that $\tilde{\sectwor} = 1$ is impossible if $l=1$, since $\sectwor < l$. Let us assume that $0 < \tilde{\sectwor} < 1$ and prove that, in this case, $h_1(x) > 0$ for all $x > 1$.

The function $h_1(x)$ is in fact a polynomial in $x^{1/[\KL : \ks]}$. We have
\[n\tilde{\sectwor} + n + 1 > \begin{cases}(n+1)\tilde \sectwor + n\\ n+1\end{cases} > \begin{cases}(n+1)\tilde \sectwor\\ n\end{cases} > n\tilde{\sectwor}\text.\]
By Descartes' rule of signs, $h_1(x)$ has at most three positive zeros (with multiplicities). Since $h_1(1) = h_1'(1) = h_1''(1) = 0$ and $\lim_{x\to \infty}h_1(x) = \infty$, we have $h_1(x) > 0$ for $x > 1$.
\end{proof}

We can now easily finish the proof of Theorem \ref{gisatmostone}. After multiplying with a suitable element from $\ks^*$ we can assume that $\theta$ is an algebraic integer and choose $\alpha := 1$. From Proposition \ref{PropositionProductForm}, Lemmata \ref{lemmaVolumeEstimate} and \ref{lemmaConstantEstimate}, and the observation that
\[\NormK(\theta)^{\frac{n}{[\KL : \ks]}} =\prod_{P}\NormK((\theta \O_\KL)_P)^{\frac{n}{[\KL : \ks]}},\]
we immediately get that $g_\ks(\theta,n)\leq 1$. Now
$g_\ks(\theta,n)=1$ holds if and only if we have equality in Lemmata
\ref{lemmaVolumeEstimate} and \ref{lemmaConstantEstimate}. This is the
case if and only if $\theta \O_\KL=\up \down \theta \O_\KL$ and for
each Archimedean place $v$ of $\ks$ the $|\theta|_w$ for $w\mid v$ are
all equal. The condition for equality in Theorem \ref{gisatmostone} is
just a uniform reformulation of these two statements.

\section{Preliminaries on lattices}\label{section4}
In this section we establish a basic counting result for lattice points, which will be used
in the proofs of Theorem \ref{generalthm} and Theorem \ref{Thmvartheta}. 

For a vector $\vx$ in $\IR^\Da$ we write $|\vx|$ for the Euclidean length of $\vx$.
For a lattice $\Lambda$ in $\IR^\Da$ we write $\lambda_i=\lambda_i(\Lambda)$ ($1\leq i\leq \Da$)
for the successive minima of $\Lambda$ with respect to the Euclidean distance. 
\begin{definition}
Let $\M$ and $\Da>1$ be positive integers and let $L$ be a non-negative real.
We say that a set $S$ is in Lip$(\Da,\M,L)$ if 
$S$ is a subset of $\IR^\Da$, and 
if there are $\M$ maps 
$\phi_1,\ldots,\phi_M:[0,1]^{\Da-1}\longrightarrow \IR^\Da$
satisfying a Lipschitz condition
\begin{alignat}3
\label{lipcond1}
|\phi_i(\vx)-\phi_i(\vy)|\leq L|\vx-\vy| \text{ for } \vx,\vy \in [0,1]^{\Da-1}, i=1,\ldots,M\text, 
\end{alignat}
such that $S$ is covered by the images
of the maps $\phi_i$.
\end{definition}
We can now state and prove our counting result.
\begin{lemma}\label{Lemmacountinglatticepts}
Let $\Da>1$ be an integer, let $\Lambda$ be a lattice in $\IR^\Da$
with successive minima $\lambda_1,\ldots,\lambda_\Da$, and let
$a\in \{1,\ldots, \Da\}$.
Let $S$ be a bounded set in $\IR^\Da$ such that
the boundary $\partial S$ of $S$ is in Lip$(\Da,M,L)$, $S$ is contained in the zero-centered ball of radius $L$, and $\vnull \notin S$.
Then $S$ is measurable and we have
\begin{alignat*}3
\left||S\cap\Lambda|-\frac{\Vol S}{\det \Lambda}\right| \leq c_1(\Da)M
\max\left\{\frac{L^{a-1}}{{\lambda_1}^{a-1}},\frac{L^{\Da-1}}{{\lambda_1}^{a-1}{\lambda_a}^{\Da-a}}\right\}.
\end{alignat*}
The constant $c_1(\Da)$ depends only on $\Da$.
\end{lemma}
\begin{proof}
Applying \cite[Theorem 5.4]{art1} proves measurability and gives
\begin{alignat*}3
\left||S\cap\Lambda|-\frac{\Vol S}{\det \Lambda}\right| \leq c_1(\Da)M
\max_{0\leq i\leq\Da-1}\frac{L^i}{\lambda_1\cdots\lambda_i}.
\end{alignat*}
First we assume $L/\lambda_1\geq 1$. 

Then we conclude
\begin{alignat*}3
\max_{0\leq i\leq\Da-1}\frac{L^i}{\lambda_1\cdots\lambda_i}
&\leq \max_{0\leq i\leq\Da-a}\frac{L^{a-1}}{\lambda_1^{a-1}}\left(\frac{L}{\lambda_a}\right)^i\\
&= \frac{L^{a-1}}{\lambda_1^{a-1}}\max\left\{1,\frac{L^{\Da-a}}{\lambda_a^{\Da-a}}\right\}\\
&= \max\left\{\frac{L^{a-1}}{{\lambda_1}^{a-1}},\frac{L^{\Da-1}}{{\lambda_1}^{a-1}{\lambda_a}^{\Da-a}}\right\}\text.
\end{alignat*}
Next we assume $L/\lambda_1< 1$. Then we have $|S\cap\Lambda|=0$. Moreover, by Minkowski's second theorem,
\begin{alignat*}3
\frac{\Vol S}{\det \Lambda}\leq c_1(\Da)\frac{L^\Da}{\lambda_1\cdots \lambda_{\Da}}.
\end{alignat*}
Furthermore,
\begin{alignat*}3
\frac{L^\Da}{\lambda_1\cdots \lambda_{\Da}}\leq\frac{L^\Da}{\lambda_1\cdots \lambda_{\Da}}\frac{\lambda_1}{L}
=\frac{L^{\Da-1}}{\lambda_2\cdots \lambda_{\Da}}
\leq \max\left\{\frac{L^{a-1}}{{\lambda_1}^{a-1}},\frac{L^{\Da-1}}{{\lambda_1}^{a-1}{\lambda_a}^{\Da-a}}\right\}.
\end{alignat*}
 \end{proof}
 We recall the following lemma, which is a special case of \cite[Lemma
 VIII.1]{18}.
\begin{lemma}\label{linearlyindependentminima}
 Let $\Lambda$ be a lattice in $\R^\Da$. Then there exist linearly independent vectors $v_1$, $\ldots$, $v_\Da$  in $\Lambda$ such that $|v_i| = \lambda_i$ for $1 \leq i \leq \Da$.
\end{lemma}

\begin{lemma}\label{shortbasis}
 Let $\Lambda$ be a lattice in $\R^\Da$. Then there exists a basis $u_1$, $\ldots$, $u_\Da$ of $\Lambda$ such that
 \begin{equation*}
      |u_i| \leq C_0(\Da) \lambda_1^{-\Da+1}\det\Lambda\quad\text{ for $1\leq i\leq \Da$, }
 \end{equation*}
where $C_0(\Da)$ is an explicit constant depending only on $\Da$.
\end{lemma}

\begin{proof}
Let $v_1$, $\ldots$, $v_\Da$ be linearly independent vectors as in Lemma \ref{linearlyindependentminima}. 
By a lemma of Mahler and Weyl (see \cite[Lemma 8, p. 135]{18}) we obtain a basis $u_1$, $\ldots$, $u_\Da$ of 
$\Lambda$ such that $|u_i| \leq \max\{1,\Da/2\}\lambda_i$. Observing that $|u_i| \leq |u_1| \cdots |u_\Da| / \lambda_1^{\Da-1}$, 
the lemma follows from Minkowski's second theorem.
\end{proof}

The following result will be used only for the proof of Theorem \ref{Thmvartheta} in Section \ref{section9}.
\begin{lemma}\label{latticeminimaestimates}
Let $\Lambda_1$ and $\Lambda_2$ be lattices in $\R^d$, and consider the lattice 
$\Lambda:=\Lambda_1\times\Lambda_2$ in $\R^{2d}$.
Then we have 
\begin{alignat*}1
\lambda_1(\Lambda)&=\min\{\lambda_1(\Lambda_1),\lambda_1(\Lambda_2)\},\\
\lambda_{d+1}(\Lambda)&\geq\max\{\lambda_1(\Lambda_1),\lambda_1(\Lambda_2)\}.
\end{alignat*}
\end{lemma}
\begin{proof}
The first assertion is obvious. For the second assertion we choose, by Lemma \ref{linearlyindependentminima}, $d+1$ linearly independent elements $v_j=(w_j^{(1)},w_j^{(2)})\in \Lambda$ ($1\leq j\leq d+1$)
with $|v_j|=\lambda_j$. Clearly, not all of them can lie in $\R^d\times \{\vnull\}$, and similarly
not all of them can lie in $\{\vnull\} \times\R^d$. Suppose $v_{j_1}\notin \R^d\times \{\vnull\}$
and $v_{j_2}\notin \{\vnull\} \times\R^d$. Hence $|v_{j_1}|\geq |w_{j_1}^{(2)}|\geq \lambda_1(\Lambda_2)$ and
$|v_{j_2}|\geq |w_{j_2}^{(1)}|\geq \lambda_1(\Lambda_1)$. This proves the lemma. 
\end{proof}

\section{Adelic Lipschitz heights}\label{section5}
In \cite{1} Masser and Vaaler have introduced what one may call Lipschitz heights on $\P^n(\KL)$.
This notion generalizes  the absolute Weil height and allows so-called Lipschitz distance functions instead of just the maximum norm 
at the Archimedean places. Nonetheless, this notion is sometimes too rigid, as one often also needs modification at a finite number 
of non-Archimedean places. This leads naturally to the concept of adelic Lipschitz heights, introduced in \cite{art1}.

\subsection{Adelic Lipschitz systems on a number field}\label{subsecdefALS}
Let $\KL$ be a number field and
recall that $M_\KL$ denotes the set of places of $\KL$, and that for every place $w$ we have fixed a
completion $\KL_w$ of $\KL$ at $w$. We write $d_w=[\KL_w:\Q_w]$ for the local degree, where 
$\Q_w$ denotes the completion of $\Q$ with respect to the unique place of $\Q$ that extends to $w$.
The value set of $w$, $\Gamma_w:=\{|\alpha|_w:\alpha \in \KL_w\}$
is equal to $[0,\infty)$ if $w$ is Archimedean,
and to
\begin{alignat*}3
\{0,(\NormK\il P_w)^{0},(\NormK\il P_w)^{\pm 1/d_w},(\NormK\il P_w)^{\pm 2/d_w},\ldots\}
\end{alignat*}
(topologized by the trivial topology) if $w$ is a non-Archimedean
place corresponding to the prime ideal $\il P_w$ of $\O_\KL$.  For $w
\mid \infty$ we identify $\KL_w$ with $\IR$ or $\IC$, respectively,
and we identify $\IC$ with $\IR^2$.
\begin{definition}\label{defALS}
An adelic Lipschitz system $\en$ on $\KL$ (of dimension $n$) is
a set of continuous maps
\begin{alignat}3
\label{Abb1}
N_w: \KL_w^{n+1}\rightarrow \Gamma_w \quad w \in M_\KL
\end{alignat}
such that for $w \in M_\KL$ we have
\begin{alignat*}3
(i)&\text{ } N_w({\vz })=0 \text{ if and only if } {\vz} ={\vnull},\\
(ii)&\text{ } N_w(a {\vz})=|a|_w N_w({\vz}) \text{ for all
$a\in\KL_w$ and all ${\vz}\in\KL_w^{n+1}$},\\
(iii)&\text{ if $w \mid \infty$: }\{{\vz}:N_w({\vz})=1\} \text{
is in $Lip(d_w(n+1),\M_w,L_w)$ for some $\M_w, L_w$},\\
(iv)&\text{ if $w \nmid \infty$: }N_w({\vz_1}+{\vz_2})
\leq \max\{N_w({\vz}_1),N_w({\vz}_2)\} \text{ for all 
${\vz}_1$, ${\vz}_2\in\KL_w^{n+1}$}.
\end{alignat*}
Moreover, we assume that the equality of functions
\begin{alignat}3
\label{Nvmaxnorm}
N_w(\vz)=\max\{|z_0|_w,\ldots,|z_n|_w\}
\end{alignat}
holds for all but a finite number
of $w \in M_\KL$.
\end{definition}

If we consider only the functions $N_w$ for $w\mid\infty$
then we get a Lipschitz system (of dimension $n$)
in the sense of Masser and Vaaler \cite{1}.

For all $w\in M_\KL$ there are  $c_w\leq 1$ in the value group
$\Gamma_w^*=\Gamma_w\backslash\{0\}$ with
\begin{alignat}3\label{normequivalence}
c_w\max\{|z_0|_w,\ldots,|z_n|_w\}\leq N_w({\vz})\leq c_w^{-1}\max\{|z_0|_w,\ldots,|z_n|_w\}
\end{alignat}
for all $\vz=(z_0,\ldots,z_n)$ in $\KL_w^{n+1}$.
Due to (\ref{Nvmaxnorm}) we can and will assume that 
\begin{alignat}3\label{cw=1}
c_w=1 
\end{alignat}
for all but a finite number of places $w$. 
We define
\begin{alignat}3
\label{defcfin}
C^{fin}_{\en}&:=\prod_{w\nmid \infty}c_w^{-\frac{d_w}{[\KL:\Q]}}\geq 1,
\end{alignat}
and
\begin{alignat}3
\label{defcinf}
C^{inf}_{\en}&:=\max_{w\mid \infty}\{c_w^{-1}\}\geq 1.
\end{alignat}
For a prime ideal ${\il P}$ of $\O_\KL$ we write $v_{\il P}$ for the
corresponding valuation on $\KL$, normalized by $v_{\il P}(\KL^*)=\Z$.
For a nonzero fractional ideal $\A$ of $\KL$ and a non-Archimedean
place $w$ of $\KL$, associated to the prime $\il P$, we define
\begin{alignat*}1
|\A|_w:=\NormK(\il P)^{-v_{\il P}(\A)/d_w}\text,
\end{alignat*}
so that $|\alpha|_w = |\alpha\O_\KL|_w$ for $\alpha \in \KL^*$. For $w\in M_\KL$
let $\sigma_w$ be the canonical embedding of $\KL$ in $\KL_w$,
extended component-wise to $\KL^{n+1}$. 
For any nonzero $\bbalf \in \KL^{n+1}$, let
$\iN(\bbalf)$ be the unique fractional ideal of $\KL$ defined by 
\begin{alignat*}1
N_w(\sigma_w\bbalf) = |\iN(\bbalf)|_w 
\end{alignat*}
for all non-Archimedean $w \in M_\KL$,
and we set by convention  $\iN({\bf 0}):=\{0\}$.

Moreover, set
\begin{alignat*}1
\O_\KL(\bbalf) := \balpha_0 \O_\KL + \cdots + \balpha_n \O_\KL,
\end{alignat*}
so that  $\O_\KL(\bbalf)$ is simply  $\iN(\bbalf)$ for any adelic Lipschitz system with (\ref{Nvmaxnorm}) for all finite  places.
Now by (\ref{normequivalence}) we get 
\begin{alignat}1\label{normequivalence2}
c_w\max\{|\balpha_0|_w,\ldots,|\balpha_n|_w\}\leq |\iN(\bbalf)|_w \leq c_w^{-1}\max\{|\balpha_0|_w,\ldots,|\balpha_n|_w\}.
\end{alignat}
Recall that $c_w=1$ up to finitely many exceptions and let
\begin{equation*}
F_\en := \{ \il A : \text{$\il A$ nonzero fractional ideal of $\KL$ and } c_w \leq |\il A|_w \leq c_w^{-1} \text{ for all }w\nmid\infty\}\text. 
\end{equation*}
By unique factorization of fractional ideals, $F_\en$ is finite. Moreover, for any $\bbalf\in \KL^{n+1}$, we have
\begin{equation}\label{idealcomparing}
\iN(\bbalf)=\O_\KL(\bbalf)\F(\bbalf)
\end{equation}
for some $\F(\bbalf)\in F_\en$. 
Taking the product in (\ref{normequivalence2}) over all finite places with multiplicities $d_w$ shows that 
\begin{alignat}1\label{normestimate}
{C_\en^{fin}}^{-[\KL:\Q]} \NormK\O_\KL(\bbalf)\leq \NormK\iN(\bbalf)\leq {C_\en^{fin}}^{[\KL:\Q]}  \NormK\O_\KL(\bbalf).
\end{alignat}

\subsection{Adelic Lipschitz heights on $\mathbb{P}^n(\KL)$}
Let $\en$ be an adelic Lipschitz system on $\KL$
of dimension $n$. 
Then the height $\hen$ 
on $\KL^{n+1}$ is defined by
\begin{alignat*}3
\hen(\bbalf):=\prod_{w \in M_\KL} N_w(\sigma_w(\bbalf))^{\frac{d_w}{[\KL:\Q]}}.
\end{alignat*}
Thanks to the product formula
and $(ii)$ from Definition \ref{defALS} $\hen(\bbalf)$ is invariant under scalar multiplication by elements of $\KL^*$.
Therefore $\hen$ is well-defined on $\IP^n(\KL)$ by setting
\begin{alignat*}3
\hen(P):=\hen(\bbalf),
\end{alignat*}
where $P=(\balpha_0:\cdots:\balpha_n) \in \IP^n(\KL)$ and $\bbalf=(\balpha_0,\ldots,\balpha_n) \in \KL^{n+1}$.
We note that by (\ref{normequivalence}), (\ref{defcfin}) and (\ref{defcinf}) we have 
\begin{alignat}1\label{heightinequality}
(C_\en^{fin}C_\en^{inf})^{-1}H(P)\leq\Hen(P)\leq C_\en^{fin}C_\en^{inf}H(P), 
\end{alignat}
where $H(P)$ denotes the projective absolute multiplicative Weil height of $P$.
Hence, by Northcott's theorem, $\{P\in \IP^n(\KL): \hen(P)\leq X\}$ is a finite set for
each $X$ in $[0,\infty)$.

\section{The general theorem}\label{section6}
Let $\ks \subseteq \KL$ be number fields and let $\en$ be an adelic Lipschitz system of dimension $n$ on $\KL$.
Recall that the functions $N_w$, $n$, and $\KL$ are all part of the data of $\en$. 
From $\en$ we obtain an adelic Lipschitz height $\Hen$ on $\P^n(\KL)$. Our goal in this section is to derive an asymptotic formula for the counting function
\[N_\en(\P^n(\ks),X) := |\{P \in \mathbb P^n(\ks): \Hen(P) \leq X\}|\text.\]
Let us set some necessary notation first.
For each Archimedean place $v$ of $\ks$ we define a function $N_v$ on $\ks_v^{n+1}$ by
\begin{alignat}3\label{N_v}
N_v(\vz):=\prod_{w\mid v}N_w(\vz)^{\frac{d_w}{d_v[\KL:\ks]}}.
\end{alignat}
Let $\en'=\en'(\en,\ks)$ be the collection of functions $N_v$, where $N_v$ is as in (\ref{N_v}) if $v$ is an Archimedean place of $\ks$ and 
\[N_v(\vz) := \max\{|z_0|_v, \ldots, |z_n|_v\}\]
if $v$ is a non-Archimedean
place of $\ks$. From now on we assume that $\en'$ is an adelic Lipschitz system (of dimension $n$) on $\ks$
(the conditions $(i), (ii)$ and $(iv)$ are automatically satisfied but $(iii)$ may possibly fail).
Hence there exists a positive integer $M_{\en'}$ and a positive real number $L_{\en'}$ such that
the sets defined by  $N_v(\vz)=1$ lie in Lip$(d_v(n+1),M_{\en'},L_{\en'})$ for all Archimedean places $v$ of $\ks$.
The sets defined by $N_v(\vz)<1$ are measurable and have a finite, positive volume, which we denote by $V_v$, and set
\begin{alignat}3\label{Ven'}
V_{\en'}:=\prod_{v\mid \infty} V_v.
\end{alignat}
We denote by $\sigma_1,\ldots,\sigma_d$ 
the embeddings from $\ks$ to $\IR$ or
$\IC$ respectively, ordered such that
$\sigma_{r+s+i}=\overline{\sigma}_{r+i}$ for 
$1\leq i \leq s$.
We define
\begin{alignat}3
\label{sigd}
&\sigma:\ks\longrightarrow \IR^r\times\IC^s = \R^d\\
\nonumber&\sigma(\balpha)=(\sigma_1(\balpha),\ldots,\sigma_{r+s}(\balpha))
\end{alignat}
and extend $\sigma$ component-wise to get a map
\begin{alignat}3
\label{sigD}
\sigma:\ks^{n+1}\longrightarrow \IR^{\Da},
\end{alignat}
where $\Da=d(n+1)$.

For nonzero fractional ideals $C$ of $\ks$, and $\il{D}$ of $\KL$, we define the following subsets of $\R^\Da$:
\begin{align*}
\Lambda^*_C(\il D) :&= \{\sigma(\bbalf): \bbalf \in \ks^{n+1}\text{, }\O_\ks(\bbalf) = C\text{, }\iN(\bbalf) = \il{D}\}\text,\\
\Lambda_C(\il D) :&= \{\sigma(\bbalf): \bbalf \in \ks^{n+1}\text{, }\O_\ks(\bbalf) = C\text{, }\iN(\bbalf) \subseteq \il{D}\}\text,\\
\Lambda(\il D) :&= \{\sigma(\bbalf): \bbalf \in \ks^{n+1}\text{, }\iN(\bbalf) \subseteq \il{D}\}\text.
\end{align*}
Note that by (\ref{idealcomparing}) we have
\begin{alignat}1\label{SCestimate}
\il D\in \up C F_\en
\end{alignat}
whenever $\Lambda^*_C(\il D)$ is non-empty, where $\up C F_\en$ denotes the finite set of fractional ideals of the form $\up C \F$ with $\F\in F_\en$.

Let $\mathcal{R}$ be a set of integral representatives for the class group Cl$_\ks$. For any $C \in \mathcal{R}$, we choose a finite set $S_C$ of nonzero fractional ideals of $\KL$ such that
\begin{equation*}
\text{$S_C$ contains all $\il{D}$ with $\Lambda^*_C(\il D) \neq \emptyset$.}
\end{equation*}
Moreover, we choose a finite set $T$ in the following way. For any $\il D \in S_C$, let $T_{C, \il D}$ be the set of all nonzero ideals $\il A$ of $\O_\KL$ such that $\Lambda_C(\il A\il D) \neq \emptyset$. This set is finite, since, similar as above, we have
$\il A\il D \il E\in \up C F_\en$ for some ideal $\il E$ of $\O_\KL$ whenever  $\Lambda_C(\il A\il D) \neq \emptyset$. Then we choose $T$ to be any finite set of nonzero ideals of $\O_\KL$ such that
\begin{equation*}
\text{$T$ contains all the sets $T_{C, \il D}$ for $C \in \mathcal{R}$ and $\il D \in S_C$.}  
\end{equation*}
We define
\begin{equation}\label{general_formula}
g_\ks^\en := \sum_{C \in \mathcal{R}}\sum_{\il D \in S_C}\sum_{\il A \in T}\mu_\KL(\il A)\sum_{E \unlhd \O_\ks}\mu_\ks(E)\frac{\NormK\il D^{\frac{n+1}{[\KL : \ks]}}}
{\det \Lambda(\il A\il D, CE)}\text,
\end{equation}
where
\[\Lambda(\il A \il D, CE) = \Lambda(\il A\il D) \cap \sigma((CE)^{n+1})\text.\]
Note that the infinite sum in (\ref{general_formula}) taken over all nonzero ideals $E$ converges absolutely, as $\det \Lambda(\il A\il D, CE)\geq (2^{-s}\Normk CE)^{n+1}$. 
Although $g_\ks^\en$ seems to depend on the choice of $\mathcal{R}$, $S_C$ and $T$, we will see that this is actually not the case.
Of course, one could impose a minimality condition to render the choice of the sets $S_C$ and $T$ unique, but 
for the calculation of $g_\ks^\en$ it is convenient to have more flexibility for the choices of these sets. 
From Theorem \ref{generalthm}, (\ref{heightinequality}), and Schanuel's theorem it will follow that $g_\ks^\en>0$. 

Finally, we define 
\begin{alignat}1\label{Aen}
A_\en:=A_\en(\ks):=|F_\en|M_{\en'}^{r+s}((L_{\en'}+C_{\en'}^{inf})C_{\en}^{fin})^{d(n+1)-1}.
\end{alignat}
We can now state the theorem.
\begin{theorem}\label{generalthm}
Let $\ks\subseteq \KL$ be number fields, $d:=[\ks:\Q]$, let $\en$ be an adelic Lipschitz system (of dimension $n$) on $\KL$,  
and suppose that $\en'=\en'(\en,\ks)$ is an adelic Lipschitz system (of dimension $n$) on $\ks$. 
Then, as $X\geq 1$ tends to infinity,
we have
\begin{alignat*}1
N_\en(\P^n(\ks),X)=\omega_\ks^{-1}(n+1)^{r+s-1}R_\ks V_{\en'} g_\ks^\en X^{d(n+1)}+O(|T|A_\en X^{d(n+1)-1}\mathfrak{L}),
\end{alignat*}
where $\mathfrak{L}=1+\log(C_{\en'}^{inf}C_{\en}^{fin}X)$ if $(n,d)=(1,1)$ and $\mathfrak{L}=1$ otherwise.
The implicit constant in the $O$-term depends only on $\ks$ and on $n$.
\end{theorem}
The hypothesis of $\en'$ being an adelic Lipschitz system is a minor one. For instance, this hypothesis is certainly fulfilled when the functions $N_w$ of $\en$ are norms,
as we shall prove in the appendix (see Lemma \ref{en-norms}).

\section{Proof of Theorem \ref{generalthm}}\label{section7}
The proof of Theorem \ref{generalthm} makes frequent use of arguments from \cite{1} and \cite{art1} (some of which can be traced back to \cite{25}, 
or even to Dedekind and Weber).

Let $q:=r+s-1$, $\Sigma$ the hyperplane in $\IR^{q+1}$
defined by $x_1+\cdots+x_{q+1}=0$ and 
$\vdelta=(d_1,\ldots,d_{q+1})$ with 
$d_i=1$ for $1\leq i \leq r$ and $d_i=2$ 
for $r+1\leq i \leq r+s=q+1$.
The map $l(\eta):=(d_1\log|\sigma_1(\eta)|,\ldots,d_{q+1}\log|\sigma_{q+1}(\eta)|)$
sends $\ks^*$ to $\IR^{q+1}$. For $q>0$ the image of the unit group
$\O_\ks^*$ under $l$ is a lattice in $\Sigma$ with
determinant $\sqrt{q+1}R_\ks$.\\

We now define a set $S_F(\TE)$ using our adelic Lipschitz system $\en'$ on $\ks$. 
Let $F$ be a bounded set in $\Sigma$ and for $\TE>0$ let $F(\TE)$ be the vector sum
\begin{alignat}3
\label{vecsum1}
F(\TE):=F+\vdelta(-\infty,\log \TE].
\end{alignat}
We denote by $\exp$ the diagonal exponential map
from $\IR^{q+1}$ to $(0,\infty)^{q+1}$.
Any embedding $\sigma_i$ ($1\leq i\leq q+1$) corresponds to an Archimedean place $v$,
and thus gives rise to one of our Lipschitz distance functions $N_i:=N_v$ from $\en'$.
We use variables ${\bf z}_1,\ldots,{\bf z}_{q+1}$ with ${\bf z}_i$
in $\IR^{d_i(n+1)}$. Exactly as in \cite{1} we define $S_F(\TE)$ in $\IR^\Da$
for $\Da=\sum_{i=1}^{q+1}d_i(n+1)=d(n+1)$ as the set of  all
${\bf z}_1,\ldots,{\bf z}_{q+1}$ such that
\begin{alignat}3
\label{inkl1}
(N_1({\bf z}_1)^{d_1},\ldots,N_{q+1}({\bf z}_{q+1})^{d_{q+1}}) \in \exp(F(\TE)).
\end{alignat}
We note that 
\begin{alignat}3\label{0notinSF}
\vnull \notin S_F(\TE).
\end{alignat}
Using $(ii)$ from Definition \ref{defALS} it is easily seen that $S_F(\TE)$ is homogeneously expanding, i.e.,
\begin{alignat}3\label{homexp}
S_F(\TE)=\TE S_F(1).
\end{alignat}
Moreover, if $F$ lies in a zero-centered ball of radius $r_F$ then
\begin{alignat*}3
S_F(\TE)&\subseteq \{(\vz_1,\ldots,\vz_{q+1}): N_i(\vz_i)\leq \exp(r_F)\TE  \text{ for }1\leq i\leq q+1\}.
\end{alignat*}
The latter set lies in the the zero-centered ball of radius $\sqrt{\Da}C_{\en'}^{inf} \exp(r_F) \TE$, and thus
\begin{alignat}3\label{SFnormboundallg}
S_F(\TE)&\subseteq B_0(\sqrt{\Da}C_{\en'}^{inf} \exp(r_F) \TE).
\end{alignat}
Note that for $q=0$ we automatically have 
$F=\{0\}$, and our set $S_F(\TE)$ is precisely the set defined by $N_1(\vz)\leq \TE$.

We now specify our set $F$ when $q>0$. We choose a basis $u_1$, $\ldots$, $u_q$ of the lattice $l(O_\ks^*)$ as in Lemma \ref{shortbasis}. Set $F:=[0,1)u_1+\cdots+[0,1)u_q$. So $F$ is measurable of ($q$-dimensional) volume 
\begin{alignat}1\label{FVol}
\Vol(F)=\sqrt{q+1} R_\ks 
\end{alignat}
(and this remains true for $q=0$). 
From the argument in \cite{art1} following (8.2), we see that $\lambda_1(l(\O_\ks^*))\geq c_d$ for some positive constant $c_d$ depending only on $d$. With the estimate from Lemma \ref{shortbasis}, we get
\begin{alignat}1\label{uibound}
|u_i|\leq C_0(q)c_d^{-q+1}\Vol(F) \leq C_d R_\ks, \quad (1\leq i\leq q) 
\end{alignat}
for some positive constant
$C_d$ depending only on $d$. Note that $F$ lies in the zero centered ball of radius 
$q C_d R_\ks$, 
and this remains trivially true for  $q=0$.
Therefore by (\ref{SFnormboundallg})
\begin{alignat}1
\label{SFnormbound}
S_F(\TE)&\subseteq B_0(\kappa \TE),
\end{alignat} 
where
\begin{alignat}3\label{kappa}
\kappa&:=\sqrt{\Da}C_{\en'}^{inf} \exp(q C_d R_\ks).
\end{alignat}
\begin{lemma}\label{SFLip}
There exists a constant $c_{\ks}(n)$ depending only on $\ks$ and $n$, a positive integer $\widetilde{M}$, and a positive real $\widetilde{L}$ with
$\widetilde{M}\leq c_\ks(n) M_{\en'}^{q+1}$, $\widetilde{L}\leq c_\ks(n)(L_{\en'}+C_{\en'}^{inf})$, such that
\begin{alignat}1\label{SFLip_eq}
\partial S_F(\TE) \in \Lip(\Da,\widetilde{M},\widetilde{L}\TE)\text{ and }S_F(\TE)\subseteq B_0(\widetilde{L}\TE)\text. 
\end{alignat} 
\end{lemma}
\begin{proof}
The second part follows immediately from \eqref{SFnormbound} and \eqref{kappa}.

Let us now prove the first part. For $q=0$ our set $S_F(\TE)$ is precisely the set defined by $N_v(\vz)\leq \TE$, where
$v$ is the single Archimedean place of $\ks$. So the boundary of $S_F(\TE)$ is the set $\{\vz:N_v(\vz)=\TE\}=\TE\{\vz:N_v(\vz)=1\}$.
By assumption $\en'$ is an adelic Lipschitz system, and thus the latter set lies in Lip$(\Da,M_{\en'},L_{\en'}\TE)$. This proves the lemma for $q=0$.

Suppose now that $q\geq 1$. 
Then we can find $2q$ linear maps $\psi_i:[0,1]^{q-1}\rightarrow \Sigma$ parameterizing $\partial F$ that, because of (\ref{uibound}), will satisfy a 
Lipschitz condition with 
constant $(q-1)C_d R_\ks$ (for $q=1$ this is simply interpreted as $|\partial F|\leq 2$). 
The claim now follows from \cite[Lemma 7.1]{art1} by a simple computation. 
\end{proof}
We conclude from \cite[Lemma 4]{1}, (\ref{FVol}), and (\ref{homexp}) that
$S_F(\TE)$ is measurable and has volume
\begin{alignat}1\label{SFVol}
\Vol S_F(\TE)=(n+1)^q R_\ks V_{\en'}  \TE^{\Da}.
\end{alignat}

\begin{lemma}\label{mainlemma}We have
\begin{alignat*}1
N_\en(\P^n(\ks),X) = \omega_\ks^{-1}\sum_{C \in \mathcal{R}}\sum_{\il D \in S_C}|\Lambda_C^*(\il D) \cap S_F(X \NormK\il D^{1/[\KL : \Q]})|\text.
\end{alignat*}
\end{lemma}
\begin{proof}
Let $P\in \IP^n(\ks)$ with homogeneous coordinates
$(\balpha_0,\ldots,\balpha_n)=\bbalf \in \ks^{n+1}\backslash \{\bf{0}\}$.
Recall the definition of the adelic Lipschitz system $\en'$.  
The functions $N_v$ (or $N_i$) will denote those associated with $\en'$, whereas 
$N_w$ will denote a function associated with the adelic Lipschitz system $\en$ on $\KL$.

Now
\begin{alignat}3
\label{NvId1}
\iNprime(\bbalf)=\O_\ks(\bbalf)  
\end{alignat}
Suppose $\varepsilon \in \ks^*$. Then we have
\begin{alignat*}3
\label{NvId1}
\iNprime(\varepsilon\bbalf)=\varepsilon\iNprime(\bbalf).  
\end{alignat*}
Hence the ideal class of $\iNprime(\bbalf)$ is independent
of the coordinates $\bbalf$ we have chosen.
In particular, we can choose $\bbalf$ such that
$\iNprime(\bbalf)=C$ for some unique $C$ in $\mathcal{R}$. Thus,
$\bbalf$ is unique up to scalar multiplication by units $\eta$, and moreover, $\iN(\bbalf):=\D\in S_C$.
The set $F(\infty)=F+\R \vdelta$ is
a fundamental set of $\R^{q+1}$ under the 
action of the additive subgroup $l(\O_\ks^*)$.
Because of Definition \ref{defALS}, $(ii)$ we have  
\begin{alignat*}3
\log N_i(\sigma_i(\eta\bbalf))^{d_i}
=\log N_i(\sigma_i\bbalf)^{d_i}+d_i\log|\sigma_i\eta|
\end{alignat*}
for $1\leq i \leq q+1$.
Hence, there exist exactly $\omega_\ks$ representatives
$\bbalf$ of $P$ with 
\begin{alignat*}3
(d_1\log N_1(\sigma_1\bbalf),\ldots,d_{q+1}\log N_{q+1}(\sigma_{q+1}\bbalf))
\in F(\infty).
\end{alignat*}
But the above is equivalent with  
\begin{alignat*}3
(N_1(\sigma_1\bbalf)^{d_{1}},\ldots,N_{q+1}(\sigma_{q+1}\bbalf)^{d_{q+1}})
\in \exp(F(\infty)).
\end{alignat*}
Furthermore
\begin{alignat*}3
\exp(F(\TE_0))=\{(X_1,\ldots,X_{q+1})\in \exp(F(\infty)): 
X_1\cdots X_{q+1}\leq \TE_0^d\}.
\end{alignat*}
Hence, for all $\omega_\ks$ representatives $\bbalf$ of $P$ as above,
the inequality
\begin{alignat*}3
\prod_{v\mid \infty}N_v(\sigma_v(\bbalf))^{d_v/d}=\prod_{v\mid \infty}\prod_{w\mid v}N_w(\sigma_w(\bbalf))^{d_w/[\KL:\Q]}\leq \TE_0 
\end{alignat*}
is equivalent to
\begin{alignat*}3
\sigma\bbalf \in S_F(\TE_0).
\end{alignat*}
On the other hand, 
\begin{alignat*}3
\prod_{w\nmid \infty}N_w(\sigma_w(\bbalf))^{d_w/[\KL:\Q]}=\NormK \iN(\bbalf)^{-1/[\KL:\Q]}=\NormK \D^{-1/[\KL:\Q]}.
\end{alignat*}
As
\begin{alignat*}3
\hen(P)=\prod_{v\mid \infty}\prod_{w\mid v}N_w(\sigma_w(\bbalf))^{d_w/[\KL:\Q]}\prod_{w\nmid \infty}N_w(\sigma_w(\bbalf))^{d_w/[\KL:\Q]},
\end{alignat*}
the claim follows. 
\end{proof}

\begin{lemma}\label{NafterMI}We have
\begin{alignat*}1
&N_\en(\P^n(\ks),X) =\\ &\omega_\ks^{-1}\sum_{C \in \mathcal{R}}\sum_{\il D \in S_C}
\sum_{\il A \in T}\mu_\KL(\il A)\sum_{E \unlhd \O_\ks}\mu_\ks(E)|\Lambda(\il A \il D, CE) \cap S_F(X \NormK\il D^{1/[\KL : \Q]})|,
\end{alignat*}
where $E$ runs over all nonzero ideals of $\O_\ks$.
\end{lemma}
\begin{proof}
We start off from Lemma \ref{mainlemma} and we apply M\"obius inversion twice to get rid of the two coprimality conditions $_C$ and $^*$.

Directly from the definition we get
\begin{alignat*}1
\Lambda_C(\il A \il D) = \bigcup_{\il B}\Lambda^*_C(\il A \il B \il D)\text,
\end{alignat*}
where $\il B$ runs over all nonzero ideals of $\O_\KL$. This is clearly a disjoint union. Note that $\Lambda^*_C(\il A \il B \il D)\neq \emptyset$ only when $\il A \il B \il D$ lies in the finite set $S_C$. 
M\"obius inversion leads then to
\begin{alignat*}1
|\Lambda_C^*(\il D) \cap S_F(X \NormK\il D^{1/[\KL : \Q]})|&=
\sum_{\il A}\mu_\KL(\il A)\sum_{\il B} |\Lambda_C^*(\il A\il B \il D) \cap S_F(X \NormK\il D^{1/[\KL : \Q]})|\\
&=\sum_{\il A}\mu_\KL(\il A)|\Lambda_C(\il A \il D) \cap S_F(X \NormK\il D^{1/[\KL : \Q]})|,
\end{alignat*}
where the sums run over all nonzero ideals in $\O_\KL$.
Next note that by definition of $T_{C,\il D}$ we have $\Lambda_C(\il A \il D)=\emptyset$ whenever $\il A\notin T_{C,\il D}$.
As $T_{C,\il D}\subseteq T$ we can restrict the last sum to $\il A\in T$ and we get
\begin{alignat*}1
|\Lambda_C^*(\il D) \cap S_F(X \NormK\il D^{1/[\KL : \Q]})|=\sum_{\il A\in T}\mu_\KL(\il A)|\Lambda_C(\il A \il D) \cap S_F(X \NormK\il D^{1/[\KL : \Q]})|.
\end{alignat*}

We now deal with the second coprimality condition $_C$. Also directly from the definition we get
\begin{equation*}
\Lambda(\il A \il D, EC)=\Lambda(\il A \il D) \cap \sigma((EC)^{n+1}) = \bigcup_{B \unlhd \O_\ks}\Lambda_{ECB}(\il A \il D) \cup \{0\}\text.
\end{equation*}
Again, $B$ runs over all nonzero ideals of $\O_\ks$ and the union is disjoint.
As $\sigma((EC)^{n+1})$ is a lattice and $ S_F(X \NormK\il D^{1/[\KL : \Q]})$ is bounded we conclude from the latter equality that
$\Lambda_{ECB}(\il A \il D)\cap S_F(X \NormK\il D^{1/[\KL : \Q]})$ is empty for all but finitely many $B$.
M\"obius inversion and \eqref{0notinSF} lead therefore to
\begin{alignat*}1
&|\Lambda_C(\il A \il D) \cap S_F(X \NormK\il D^{1/[\KL : \Q]})|\\&= \sum_{E \unlhd \O_\ks}\mu_\ks(E)\sum_{B \unlhd \O_\ks}|\Lambda_{ECB}(\il A \il D)\cap S_F(X \NormK\il D^{1/[\KL : \Q]})|\\
&= \sum_{E \unlhd \O_\ks}\mu_\ks(E)|\Lambda(\il A \il D, CE) \cap S_F(X \NormK\il D^{1/[\KL : \Q]})|.
\end{alignat*}
In view of Lemma \ref{mainlemma} this proves the claim.
\end{proof}

We choose a positive real $\nb$ such that for any $C\in \mathcal{R}$ and any $\il D\in S_C$ 
\begin{alignat}1\label{nb}
\nb\leq \frac{\Normk C}{\NormK(\il D)^{1/[\KL:\ks]}}.
\end{alignat}
Before we proceed note that if $S_C$ is chosen minimal for all $C\in \mathcal{R}$ (i.e. $S_C=\{\iN(\bbalf): \bbalf \in \ks^{n+1}, \O_\ks(\bbalf)=C\}$)
then it follows from (\ref{normestimate}) that we can choose $\nb={C_\en^{fin}}^{-d}$, and moreover, $|S_C|\leq |F_\en|$. 
\begin{lemma}\label{appliedlatticeestimate}
Let $\lambda_1=\lambda_1(\Lambda(\il A \il D, CE))$ be the first successive minimum of the lattice $\Lambda(\il A \il D, CE)$, and let
$\widetilde{M}$ and $\widetilde{L}$ be as in Lemma \ref{SFLip}. Then we have  
\begin{alignat*}1
|\Lambda(\il A \il D, CE) \cap S_F(X \NormK\il D^{1/[\KL : \Q]})|=
&\frac{\Vol S_F(1)\NormK\il D^{\frac{n+1}{[\KL : \ks]}}X^{\Da}}{\det \Lambda(\il A \il D, CE)}\\
+&O\left(\widetilde{M}\frac{\NormK\il D^{\frac{\Da-1}{[\KL : \Q]}}(\widetilde{L}X)^{\Da-1}}{\lambda_1^{\Da-1}}\right),
\end{alignat*}
where the constant in the $O$-term depends only on $\Da$. Moreover, with $\nb$ as in (\ref{nb}) we have
\begin{alignat*}1
\lambda_1\geq \NormK({\il D})^{1/[\KL : \Q]}(\nb\Normk(E))^{1/d}.
\end{alignat*}
And finally, with $\kappa$ as in (\ref{kappa}), if $\Normk E>(\kappa X)^d/\nb$ then 
\begin{alignat*}1
\Lambda(\il A \il D, CE) \cap S_F(X \NormK\il D^{1/[\KL : \Q]})=\emptyset.
\end{alignat*}
\end{lemma}
\begin{proof}
For the first assertion we use (\ref{0notinSF}) and apply Lemma \ref{Lemmacountinglatticepts} with $a=\Da$. Thanks to (\ref{SFnormbound}) and Lemma \ref{SFLip} the required 
conditions are satisfied, and using (\ref{homexp}) the first result drops out.

Now for the second statement we first observe that $\lambda_1$ is at least as large as the first successive minimum of the lattice $\sigma(CE)$. But it is well-known
that the latter is at least $\Normk(CE)^{1/d}$, see, e.g., \cite[Lemma 5]{1}. Now as
${\il D}\in S_C$ and by the definition of $\nb$ we get $\Normk C\geq \nb\NormK(\il D)^{1/[\KL:\ks]}$ and this yields the second assertion.

The last claim follows upon combining the above estimate for $\lambda_1$ with (\ref{0notinSF}), (\ref{SFnormbound}).   
\end{proof}
We can now conclude the proof of Theorem \ref{generalthm}. 
Let us first assume that $(n,d)\neq (1,1)$. Combining Lemma \ref{NafterMI}, Lemma \ref{appliedlatticeestimate} and (\ref{SFVol}) gives the main term
as in Theorem \ref{generalthm}. The error term is bounded by   
\begin{alignat*}1
&\sum_{C \in \mathcal{R}}\sum_{\il D \in S_C}
\sum_{\il A \in T}\sum_{E \unlhd \O_\ks} O\left(\widetilde{M}\frac{\NormK\il D^{\frac{\Da-1}{[\KL : \Q]}}(\widetilde{L}X)^{\Da-1}}{\lambda_1^{\Da-1}}\right)\\
\leq &\sum_{C \in \mathcal{R}}\sum_{\il D \in S_C}
\sum_{\il A \in T}\sum_{E \unlhd \O_\ks} O\left(\frac{\widetilde{M}(\widetilde{L}X)^{\Da-1}}{\nb^{(\Da-1)/d}\Normk E^{(n+1)-1/d}}\right)\\
\leq &\sum_{C \in \mathcal{R}}\sum_{\il D \in S_C}
\sum_{\il A \in T}O\left(\frac{\widetilde{M}(\widetilde{L}X)^{\Da-1}}{\nb^{(\Da-1)/d}}\right)\\
= & O\left(\sum_{C \in \mathcal{R}}|S_C||T|\frac{\widetilde{M}(\widetilde{L}X)^{\Da-1}}{\nb^{(\Da-1)/d}}\right)\
\end{alignat*}
This proves the Theorem in the case $(n,d)\neq (1,1)$ except that the constant in the error term is different from the one in the statement 
of the theorem. In particular, it shows that the main term is independent of the particular choice of the sets $S_C$.
However, if we choose all the sets $S_C$ to be minimal then, by the remark just after (\ref{nb}), we can choose
$\nb={C_\en^{fin}}^{-d}$, and $|S_C|\leq |F_\en|$. This, and not forgetting the definition of $\widetilde{M}$ and $\widetilde{L}$ from Lemma \ref{SFLip}, yields the desired error term.

We now assume $(n,d)=(1,1)$ (which of course means $\ks=\Q$, $\mathcal{R}=\{C\}$, $\omega_\ks=2$). 
Using also the last part of Lemma \ref{appliedlatticeestimate} we conclude
\begin{alignat*}1
&N_\en(\P^1(\Q),X) = \\
&\frac{1}{2}\sum_{\il D \in S_C}
\sum_{\il A \in T}\mu_\KL(\il A)\sum_{E \unlhd \Z \atop \Norm_\Q E\leq \kappa X/\nb}\mu_\Q(E)|\Lambda(\il A \il D, CE) \cap S_F(X \NormK\il D^{1/[\KL : \Q]})|\\
&=\frac{1}{2}\sum_{\il D \in S_C}
\sum_{\il A \in T}\mu_\KL(\il A)\sum_{E \unlhd \Z}\mu_\Q(E)
\frac{\Vol S_F(1)\NormK\il D^{\frac{2}{[\KL : \Q]}}X^{2}}{\det \Lambda(\il A \il D, CE)}\\
&+O\left(\sum_{\il D \in S_C}
\sum_{\il A \in T}\sum_{E \unlhd \Z \atop \Norm_\Q E>\kappa X/\nb}\frac{\Vol S_F(1)\NormK\il D^{\frac{2}{[\KL : \Q]}}X^{2}}{\det \Lambda(\il A \il D, CE)}\right)\\
&+
O\left(\sum_{\il D \in S_C} \sum_{\il A \in T}\sum_{E \unlhd \Z \atop \Norm_\Q E\leq \kappa X/\nb}
\frac{\widetilde{M}\NormK\il D^{\frac{1}{[\KL : \Q]}}\widetilde{L}X}{\lambda_1}\right).
\end{alignat*}
Now the first term gives the main term as before. For the second term we use Minkowski's first theorem to estimate the determinant in terms of $\lambda_1$, and then a simple 
computation using Lemma \ref{appliedlatticeestimate} and (\ref{SFnormbound}) gives the error term $O(|S_C||T|(1+\kappa X/\nb)$. 
For the last error term we use again Lemma \ref{appliedlatticeestimate}, and again a simple computation yields the error term 
\begin{equation*}
 O(|S_C||T|(\widetilde{M}\widetilde{L}/\nb)X(1+\log(\kappa X/\nb)).
\end{equation*}
To get the right error term we choose again $S_C$ to be minimal so that  we can take $\nb={C_\en^{fin}}^{-1}$, and $|S_C|\leq |F_\en|$.
This proves Theorem \ref{generalthm}.

\section{Proof of Theorem \ref{Thmthetan}}\label{section8}
In this section, we deduce Theorem \ref{Thmthetan} from Theorem \ref{generalthm}. Recall the simple facts
mentioned just before Lemma \ref{lemmaConstantEstimate}.

As mentioned after Lemma \ref{lemmaVolumeInvariant}, we can and will assume that $\theta$ is an algebraic integer. Let $\KL := \ks(\theta)$, and let $\en$ be the adelic Lipschitz system on $\KL$ of dimension $n$ defined by
\[N_w(\vz) := \max\{|z_0|_w, |\theta|_w |z_1|_w,\ldots, |\theta|_w |z_n|_w\}\text,\]
so
\begin{equation}\label{ientheta}
\iN(\bbalf) = \balpha_0 \O_\KL + \theta \balpha_1 \O_\KL + \cdots + \theta \balpha_n \O_\KL.
\end{equation}

\begin{lemma}\label{pf_thmthetan_affine_projective}
 We have
 \begin{align*}
   N(\theta\ks^n, X) = N_\en(\P^n(\ks), X) + O(X^{nd}),
 \end{align*}
where the implicit constant in the error term depends only on $\ks$, $\theta$, and $n$.
\end{lemma}

\begin{proof}
The points $\balf = (\balpha_1/\balpha_0, \ldots, \balpha_n/\balpha_0) \in \ks^n$ with $H(\theta\balf) \leq X$ are in one--to--one correspondence with the projective points $P = (\balpha_0 : \cdots : \balpha_n) \in \P^n(\ks)$ with $\balpha_0 \neq 0$ and $H_\en(P) \leq X$. 

If $n>1$ then we can apply Theorem \ref{generalthm} with $n-1$ and the adelic Lipschitz system given by the norm functions (see Lemma \ref{en-norms} in the appendix)
\begin{equation}
N_w((z_1, \ldots, z_n)) := \max\{ |\theta|_w |z_1|_w,\ldots, |\theta|_w |z_n|_w\}
\end{equation}
(with $\mathcal{R}$, $S_C$ and $T$ chosen in such a way that $|T|$ is minimal) to see that the number of such points $P$ with $\balpha_0 = 0$ is $O(X^{n d})$. This trivially remains true for $n=1$.
\end{proof}

Since the functions $N_w$ are norms, the adelic Lipschitz system $\en$ satisfies the hypothesis of Theorem \ref{generalthm}. As our choice of $\mathcal{R}$, $S_C$ and $T$ in Theorem \ref{generalthm} will depend only on $\ks$, $n$ and $\theta$, we obtain
\begin{equation}\label{sec8_general_formula}
N_\en(\P^n(\ks), X) = \omega_\ks^{-1}(n+1)^{r+s-1}R_\ks V_{\en'} g_\ks^\en X^{d(n+1)}+O(X^{d(n+1)-1}\mathfrak{L})\text,
\end{equation}

where $\mathfrak{L}:=\log(X+1)$ if $(n,d)=(1,1)$ and  $\mathfrak{L}:=1$ otherwise.
The implicit constant in the error term depends only on $\ks$, $\theta$, and $n$.

We notice that 
\begin{equation}\label{sec8_volume}
V_{\en'} = (2^{r}\pi^{s})^{n+1}V(\theta,\ks,n),
\end{equation}
with $V(\theta, \ks, n)$ as in \eqref{defV}. To prove the theorem, we need to compute $g_\ks^\en$. First we choose the sets $\mathcal{R}$, $S_C$ and $T$. Denote 
\[
D:= \down{(\theta \O_K)}\text.
\]
For $\mathcal{R}$ we choose any system of integral representatives for the class group $\Cl_\ks$ with
\begin{equation}\label{ideal_reps_rel_prime}
(C, D) = \O_\ks \text{ for all }C \in \mathcal{R}\text.
\end{equation}
We will see in Lemma \ref{SC}, \emph{(i)}, that
\begin{equation}\label{sec8_defsc}
 S_C := \{\up C(\theta \O_\KL, \up B): B \unlhd \O_\ks\text{, }B \mid D\}
\end{equation}
is a valid choice for $S_C$. For $T$, we take the finite set
\begin{equation}\label{defT}
T := \bigcup_{C \in \mathcal R}\bigcup_{\il D \in S_C}T_{C, \il D} \cup \{\il A \unlhd \O_\KL: \il A \mid \theta\O_\KL\}\text.
\end{equation}

\begin{lemma}\label{SC}\ 
 \begin{enumerate}[(i)]
 \item Let $\bbalf \in \ks^{n+1}$ with $\O_\ks(\bbalf) = C$. Then $\iN(\bbalf) \in S_C$.  
 \item Let $\il A$ be an ideal of $\O_\KL$ and $B$ an ideal of $\O_\ks$. Then $\down (\il A, \up B) = (\down \il A, B)$. 
 \item Let $B$ be an ideal of $\O_\ks$ with $B \mid D$. Then $\down (\theta \O_\KL, \up B) = B$.
 \end{enumerate}
\end{lemma}

\begin{proof}
\emph{(i)}: We have $\balpha_0 \O_\KL + \cdots + \balpha_n \O_\KL = \up \O_\ks(\bbalf) = \up C$, so
\begin{align*}
\iN(\bbalf) = \up C (\balpha_0 (\up C)^{-1} + \theta (\balpha_1 (\up C)^{-1} + \cdots + \balpha_n (\up C)^{-1})) = \up C(\balpha_0(\up C)^{-1}, \theta \O_\KL)\text.
\end{align*}
Moreover, since $\theta\O_\KL \mid \up D$, we obtain 
\[(\balpha_0 (\up C)^{-1}, \theta\O_\KL) = (\balpha_0 (\up C)^{-1}, \up D, \theta\O_\KL) = (\theta\O_\KL, \up B)\text,\]
for $B := (\balpha_0 C^{-1}, D) \mid D$.

\noindent\emph{(ii)}: Let $P$ be a prime ideal of $\O_\ks$ and $\up P = \prod_{\il P} \il P^{e_{\il P}}$ its factorization in $\O_\KL$. Then
\begin{align*}
 v_P(\down (\il A, \up B)) &= \max_{\il P}\{\lceil \min\{v_{\il P}(\il A), v_{\il P}(\il \up B)\}/e_{\il P} \rceil\}\\
 &=\max_{\il P}\{\min\{\lceil v_{\il P}(\il A)/e_{\il P} \rceil, v_P(B)\}\}\\
 &=\min\{ \max_{\il P}\{\lceil v_{\il P}(\il A)/e_{\il P} \rceil \}, v_P(B)\} = v_P((\down \il A, B))\text.
\end{align*}

\noindent\emph{(iii)}: By \emph{(ii)}, we have $\down (\theta \O_\KL, \up B)  = (D, B) = B$.
\end{proof}

The first step in our computation of $g_\ks^\en$ is to evaluate the determinant of the lattice $\Lambda(\il A \il D, CE)= \Lambda(\il A\il D) \cap \sigma((CE)^{n+1})$. 
 
\begin{lemma}\label{Lambdadeterminant}
 Let $\il A$, $B$ be nonzero ideals of $\O_\KL$ and $\O_\ks$, respectively. Then
 \[\det \Lambda(\il A, B) = (2^{-s}\sqrt{|\Delta_\ks|})^{n+1}\cdot \Normk(\down \il A \cap B)\cdot \Normk\left(\down \left(\il A(\theta \O_\KL, \il A)^{-1}\right) \cap B\right)^n\text.\]
\end{lemma}

\begin{proof}
Let $\bbalf = (\balpha_0, \ldots, \balpha_n) \in \ks^n$. Clearly, $\sigma\bbalf \in \Lambda(\il A, B)$ if and only if $\balpha_i \in B$ for all $0 \leq i \leq n$, $\balpha_0 \in \il A$, and $\theta\balpha_i \in \il A$ for all $1 \leq i \leq n$. For $\balpha_i \in \O_\ks$, we have
\[\theta \balpha_i \in \il A\quad \text{ if and only if }\quad \il A(\theta \O_\KL, \il A)^{-1} \mid \balpha_i \O_\KL\text.\]
Therefore, we obtain
\[\Lambda(\il A, B) = \sigma\left((\down \il A \cap B)\times\left(\down \left(\il A(\theta \O_\KL, \il A)^{-1}\right) \cap B\right)^n\right)\text.\]
\end{proof}

Let $\il A \in T$ and let $B$ be an ideal of $\O_\ks$ with $B \mid D$. To facilitate further notation, we define ideals $A$ and $A_1$ of $\O_\ks$ by
\begin{align}
A &= A(\il A, B)  := \down (\il A(\theta\O_\KL, \up B))\quad \text{ and }\label{sec8_defA}\\
A_1 &= A_1(\il A, B) := \down \left(\il A (\theta\O_\KL, \up B)(\theta \O_\KL, \il A\up B)^{-1}\right) \mid A\text.\label{sec8_defA1}
\end{align}
For any $\il D = \up C(\theta \O_\KL, \up B) \in S_C$ and for any nonzero ideal $E$ of $\O_\ks$ we have
\begin{align}\label{sec8_Aeq}
\Normk(\down(\il A \il D) \cap CE) = \Normk C\cdot \Normk(A \cap E)\text.
\end{align}
Clearly, we have $(\theta \O_\KL, \il A(\theta \O_\KL, \up B)) = (\theta\O_\KL, \il A \up B)$.
Furthermore, by our choice of $\mathcal{R}$ with \eqref{ideal_reps_rel_prime}, we have $(\up C, \theta \O_K) = \O_K$. Therefore, we obtain
\begin{align}\label{sec8_A1eq}
\Normk\left(\down \left((\il A \il D)(\theta \O_\KL, \il A \il D)^{-1}\right) \cap CE\right) 
= \Normk C \cdot \Normk\left(A_1 \cap E\right)\text.
\end{align}
Moreover, we have
\begin{align}\label{sec8_Deq}
\NormK\il D^{(n+1)/[\KL : \ks]} = \Normk C^{n+1}\cdot \NormK(\theta \O_\KL, \up B)^{(n+1)/[\KL:\ks]}\text.
\end{align}

\begin{lemma}\label{sec8_summand}
Let $B$ be an ideal of $\O_\ks$ with $B \mid D$, let $\il D = \up C(\theta \O_\KL, \up B) \in S_C$, let $\il A \in T$, and let $E$ be a nonzero ideal of $\O_\ks$. Then
 \[\frac{\NormK\il D^{\frac{n+1}{[\KL : \ks]}}}{\det\Lambda(\il A \il D, CE)} = (2^{-s}\sqrt{|\Delta_\ks|})^{-(n+1)}\cdot\frac{\NormK(\theta \O_\KL, \up B)^{\frac{n+1}{[\KL:\ks]}}}{\Normk(A \cap E)\cdot\Normk\left(A_1 \cap E\right)^n}\text.\]
\end{lemma}

\begin{proof}
We apply Lemma \ref{Lambdadeterminant} and use \eqref{sec8_Aeq}, \eqref{sec8_A1eq}, and \eqref{sec8_Deq}.
\end{proof}

\begin{lemma}\label{sec8_g1} We have
 \[g_\ks^\en = c_0 \sum_{B \mid D}\NormK(\theta \O_\KL, \up B)^{\frac{n+1}{[\KL : \ks]}}\sum_{\il A \in T}\mu_\KL(\il A)
\sum_{E \unlhd \O_\ks}\frac{\mu_\ks(E)}{\Normk(A \cap E)\cdot\Normk(A_1 \cap E)^n}\text,\]
where $A = A(\il A, B)$, $A_1 = A_1(\il A, B)$, and $c_0 := h_\ks 2^{s(n+1)}(\sqrt{|\Delta_\ks|})^{-(n+1)}$ and $E$ runs over all nonzero ideals of $\O_\ks$.
\end{lemma}

\begin{proof}
 Recall the definition of $g_\ks^\en$ in \eqref{general_formula}. The expression on the right-hand side in Lemma \ref{sec8_summand} does not depend on $C$. With \eqref{sec8_defsc}, a simple computation proves the lemma.
\end{proof}

The inner sum over $E$ in Lemma \ref{sec8_g1} can be handled by the following lemma.
\begin{lemma}
Let $J_1 \mid J$ be nonzero ideals of $\O_\ks$ and let
\[\xi := \sum_{E \unlhd \O_\ks}\frac{\mu_\ks(E)}{\Normk(J\cap E)\cdot \Normk(J_1\cap E)^n}\text.\]
If $J_1 \neq \O_\ks$ then $\xi = 0$. If $J_1 = \O_\ks$ then
\[\xi = \frac{1}{\zeta_\ks(n+1)\Normk(J)}\prod_{P \mid J}\frac{\Normk P^{n+1} - \Normk P}{\Normk P^{n+1} - 1}\text.\]
\end{lemma}

\begin{proof}
Let $f(E) := \mu_\ks(E)\cdot\Normk(J, E)\cdot\Normk(J_1, E)^n$. Then $f$ is multiplicative and 
\[\xi = \frac{1}{\Normk (J J_1^n)} \sum_{E \unlhd \O_\ks}\frac{f(E)}{\Normk E^{n+1}}\text.\]
Clearly, this Dirichlet series converges absolutely for all $n>0$. Let us compute its Euler product expansion.
For any prime ideal $P$ of $\O_\ks$, we have $f(P^e) = 0$ if $e \geq 2$. Moreover, $f(\O_\ks) = 1$ and
\[f(P) = \begin{cases}-\Normk P^{n+1}&\text{if $P \mid J_1$,}\\
-\Normk P&\text{if $P \mid J$ and $P \nmid J_1$,}\\
-1&\text{if $P \nmid J$.}\end{cases}\]
We obtain the formal  expansion 
\[\sum_{E \unlhd \O_\ks}\frac{f(E)}{\Normk E^{s}} = \prod_{P \mid J_1}\left(1 - \frac{\Normk P^{n+1}}{\Normk P^s}\right)\prod_{\substack{P \mid J\\P\nmid J_1}}\left(1 - \frac{\Normk P}{\Normk P^{s}}\right)\prod_{P \nmid J}\left(1 - \frac{1}{\Normk P^{s}}\right)\text.\]
Since the infinite product $\prod_{P \nmid J}\left(1 - \Normk P^{-s}\right)$ converges absolutely for $s>1$, we obtain $\xi = 0$ whenever $J_1 \neq \O_\ks$. If $J_1 = \O_\ks$ and $s = n+1$, the expression simplifies to
\[\sum_{E \unlhd \O_\ks}\frac{f(E)}{\Normk E^{n+1}} = \frac{1}{\zeta_\ks(n+1)}\prod_{P \mid J}\frac{\Normk P^{n+1} - \Normk P}{\Normk P^{n+1} - 1}\text.\] 
\end{proof}
Recall the definition of $A$ and $A_1$ from \eqref{sec8_defA} and \eqref{sec8_defA1}. We have $A_1 = \O_\ks$ if and only if $\il A(\theta \O_\KL, \up B) = (\theta \O_\KL,  \il A\up B)$, which is equivalent to $\il A(\theta \O_\KL, \up B) \mid \theta \O_\KL$, or
\begin{equation}\label{rangeforA}
\il A \mid \theta \O_\KL(\theta \O_\KL, \up B)^{-1}\text.
\end{equation}
Recall that, by \eqref{defT}, the set $T$ contains all ideals $\il A$ of $\O_\KL$ with $\il A \mid \theta\O_\KL$. Also, for every $\il A$ with \eqref{rangeforA}, we have $A = \down (\il A(\theta \O_\KL, \up B)) \mid D$. We obtain 
\[g_\ks^\en = c_1 \sum_{B \mid D}\NormK(\theta \O_\KL, \up B)^{\frac{n+1}{[\KL : \ks]}}\sum_{A \mid D}\frac{1}{\Normk A}
\prod_{P \mid A}\frac{\Normk P^{n+1} - \Normk P}{\Normk P^{n+1} - 1}s_0(A, B)\text,\]
where $c_1 := \zeta_\ks(n+1)^{-1} c_0 = h_\ks 2^{s(n+1)}\zeta_\ks(n+1)^{-1}(\sqrt{|\Delta_\ks|})^{-(n+1)}$ and
\[s_0(A, B) := \sum_{\substack{\il A \text{ with \eqref{rangeforA}}\\\down (\il A(\theta\O_\KL, \up B)) = A}}\mu_\KL(\il A)\text.\]
If $s_0(A, B)$ is not zero then there is at least one $\il A$ with 
\[A = \down (\il A(\theta\O_\KL, \up B)) \subseteq \down(\theta \O_\KL, \up B) = B\text.\]
For the last equality, we used Lemma \ref{SC}, \emph{(iii)}. We replace $A$ by $B^{-1}A$ to obtain
\[g_\ks^\en = c_1 \sum_{B \mid D}\frac{\NormK(\theta \O_\KL, \up B)^{\frac{n+1}{[\KL : \ks]}}}{\Normk B}\sum_{A \mid B^{-1}D}\frac{1}{\Normk A}
\prod_{P \mid AB}\frac{\Normk P^{n+1} - \Normk P}{\Normk P^{n+1} - 1}s(A, B)\text,\]
where
\[s(A, B) := \sum_{\substack{\il A \text{ with \eqref{rangeforA}}\\\down (\il A(\theta\O_\KL, \up B)) = AB}}\mu_\KL(\il A)\text.\]

\begin{lemma}
Let $\il J$, $\il K$ be nonzero ideals of $\O_\KL$ and $J$ a nonzero ideal of $\O_\ks$. Then $\down (\il J \il K) = J \down \il K$ if and only if
\begin{equation}\label{divisor_condition}
\il J \mid \up J (\up \down \il K)\il K^{-1}\text{ and } \il J \nmid \up(P^{-1}J)(\up \down \il K)\il K^{-1} \text{ for all prime ideals $P \mid J$.}
\end{equation}
\end{lemma}

\begin{proof}
Clearly,
\[\il J \mid \up J (\up \down \il K)\il K^{-1} \Longleftrightarrow \il J \il K \mid \up(J \down \il K) \Longleftrightarrow \down(\il J \il K) \mid J \down \il K\]
and
\[\il J \nmid \up(P^{-1}J)(\up \down \il K)\il K^{-1} \Longleftrightarrow \il J \il K \nmid \up(P^{-1}J \down \il K) \Longleftrightarrow \down(\il J \il K) \nmid (P^{-1}J)\down \il K\text.\]
\end{proof}

\begin{lemma}If $A \mid B^{-1}D$ then $s(A, B) = \mu_\ks(A)$.
\end{lemma}

\begin{proof}
By Lemma \ref{SC}, \emph{(iii)}, we have $\down (\theta \O_\KL, \up B) = B$. By the previous lemma, $\down (\il A(\theta\O_\KL, \up B)) = AB$ is equivalent to
\begin{equation}\label{divisor_condition_concrete}
\il A \mid \up A\up B(\theta\O_\KL, \up B)^{-1}\text{ and } \il A \nmid \up(P^{-1}A)\up B(\theta\O_\KL, \up B)^{-1} \text{ for all $P \mid A$.}
\end{equation}
Clearly, conditions \eqref{rangeforA} and \eqref{divisor_condition_concrete} imply
\begin{equation}\label{divisor_condition_concrete_1}
\il A \mid (\theta \O_\KL(\theta \O_\KL, \up B)^{-1}, \up A \up B(\theta \O_\KL, \up B)^{-1}) = (\theta \O_\KL(\theta \O_\KL, \up B)^{-1}, \up A)
\end{equation}
and
\begin{equation}\label{divisor_condition_concrete_2}
\il A \nmid \up(P^{-1}A) \text{ for all prime ideals }P \mid A\text.
\end{equation}
In fact, \eqref{rangeforA} and \eqref{divisor_condition_concrete} are equivalent to \eqref{divisor_condition_concrete_1} and \eqref{divisor_condition_concrete_2}. Indeed,  \eqref{divisor_condition_concrete_1} immediately implies \eqref{rangeforA} and the first part of \eqref{divisor_condition_concrete}. For the second part of \eqref{divisor_condition_concrete}, we use that every $\il A \mid \theta \O_\KL(\theta \O_\KL, \up B)^{-1}$ satisfies $(\il A,\up B(\theta\O_\KL, \up B)^{-1}) = \O_\KL$. Thus,
\[s(A, B) = \sum_{\substack{\il A \unlhd \O_\KL\\\eqref{rangeforA}\text{ and }\eqref{divisor_condition_concrete}}}\mu_\KL(\il A) = \sum_{\substack{\il A \unlhd \O_\KL\\\eqref{divisor_condition_concrete_1}\text{ and }\eqref{divisor_condition_concrete_2}}}\mu_\KL(\il A)\text.\]
By inclusion-exclusion for \eqref{divisor_condition_concrete_2}, we obtain
\[s(A, B) = \sum_{F \mid A}\mu_\ks(F)\sum_{\il A \mid (\theta\O_\KL(\theta\O_\KL, \up B)^{-1}, \up(F^{-1}A))}
\mu_\KL(\il A)\text.\]
The last sum is $1$ if $F = A$. Moreover,
\[F^{-1} A \mid B^{-1}D = \down(\theta \O_\KL)(\down(\theta \O_\KL, \up B))^{-1} \mid \down(\theta \O_\KL(\theta \O_\KL, \up B)^{-1})\text,\]
so $F \neq A$ implies that
\[(\theta\O_\KL(\theta\O_\KL, \up B)^{-1}, \up (F^{-1}A)) \neq \O_\KL\text.\]
This shows that the last sum is $0$ whenever $F \neq A$.
\end{proof}
We obtain
\[g_\ks^\en = c_1 \sum_{B \mid D}\frac{\NormK(\theta \O_\KL, \up B)^{(n+1)/[\KL : \ks]}}{\Normk B}
\sum_{A \mid B^{-1}D}\frac{\mu_\ks(A)}{\Normk A}\prod_{P \mid AB}\frac{\Normk P^{n+1} - \Normk P}{\Normk P^{n+1} - 1}\text,\]
and Theorem \ref{Thmthetan} follows by substituting this and \eqref{sec8_volume} in \eqref{sec8_general_formula}.

\section{Proof of Theorem \ref{Thmvartheta}}\label{section9}
In this section we will use not only Landau's $O$-notation but also Vinogradov's symbol $\ll$. 
All implied constants depend solely on $\ks$.
As we will encounter expressions like $\log\log X$ we assume throughout the entire section that $X\geq 3$. 
Our main task will be to prove the following proposition.
\begin{proposition}\label{properrorterm}
Suppose $p\in \mathbf{P}_\ks$. Then, as $X\geq 3$ tends to infinity, we have
\begin{alignat*}1
N(\sqrt{p}\ks^*,X)=\frac{2p^{d/2}}{p^d+1}S_\ks(1)X^{2d}+O\left(\frac{X^{2d-1}}{p^{(d-1)/2}}+X^d\log X+X^d\log p\right). 
\end{alignat*}
\end{proposition}
We choose the adelic Lipschitz system $\en$ (of dimension $1$) on $\KL:=\ks(\sqrt{p})$, defined by 
\begin{alignat*}1
N_w((z_0,z_1)):=\max\{|z_0|_w,|\sqrt{p}|_w|z_1|_w\}
\end{alignat*}
for any place $w$ of $\KL$. Recall the definition of $C_\en^{fin}$ and $C_\en^{inf}$ from (\ref{defcfin}) and (\ref{defcinf}), and note that we can take 
\begin{alignat}1\label{Cenhere}
C_\en^{fin}=C_\en^{inf}=\sqrt{p}.
\end{alignat}
The adelic Lipschitz system $\en$ on $\KL$ leads to an adelic Lipschitz system $\en'$ on $\ks$ as in Section \ref{section6}.
Note that for any Archimedean $v$ from $\ks$ and $N_v$ from $\en'$ we have $N_v((z_0,z_1))=\max\{|z_0|_v,\sqrt{p}|z_1|_v\}$.
Thus we can also take
\begin{alignat}1\label{Cen'here}
C_{\en'}^{inf}=\sqrt{p}.
\end{alignat}

\begin{lemma}\label{onetoone}
We have
\begin{alignat*}1
N(\sqrt{p}\ks^*,X)=N_\en(\P^1(k);X)-2. 
\end{alignat*} 
\end{lemma}
\begin{proof}
The map $\alpha\mapsto (1 : \alpha)$ is a one-to-one correspondence between $\ks^*$ and $\P^1(k)\backslash\{(0:1),(1:0)\}$
Moreover, $H(\sqrt{p}\alpha)=\hen((1 : \alpha))$. Hence
there is a one-to-one correspondence between $\{\alpha\in \ks^*: H(\sqrt{p}\alpha)\leq X\}$ and $\{P\in\P^1(k)\backslash\{(0:1),(1:0)\}:\Hen(P)\leq X\}$.
As $\hen((0:1))=\hen((1:0))=1$ the claim follows.
\end{proof}
We can now basically follow the proof of Theorem \ref{generalthm} using our specific adelic Lipschitz system. 
However, to get the good error terms regarding $p$ an additional idea is required.
We will use  the same notation as in Sections \ref{section6} and \ref{section7}. In particular, recall the definition of the set $S_F(\TE)$ introduced in (\ref{inkl1}). As in \eqref{ideal_reps_rel_prime}, we choose a system $\mathcal{R}$ of integral representatives for $\Cl_\ks$ such that $(C,p\O_\ks)=\O_\ks$ for all $C\in \mathcal{R}$.

\begin{lemma}\label{sctwoelements}
We can choose $S_C:=\{\up C,\sqrt{p}\up C\}$.
\end{lemma}
\begin{proof}
As in (\ref{ientheta}) we have $\iN(\bbalf)=\balpha_0\O_\KL + \sqrt{p}\balpha_1\O_\KL$. So if $\O_\ks(\bbalf)=C$ we get $\sqrt{p}\up C\subseteq \iN(\bbalf)\subseteq \up C$.
As $\sqrt{p}\O_\KL$ is a prime ideal this proves the lemma.
\end{proof}
With this choice of the sets $S_C$ we directly verify that $\nb$ from (\ref{nb}) can be chosen to be
\begin{alignat}1\label{nbp}
\nb:=p^{-d/2}.
\end{alignat}
From now on $C$ is always in $\mathcal{R}$, $\il D$ is always in $S_C$, and $\A$ will always be in $T$.
\begin{lemma}\label{setcardbounds}
We can choose $T$ such that $|T|\leq 2$. 
\end{lemma}
\begin{proof}
Recall that we may choose $T=\cup_{C\in \mathcal{R}}\cup_{\il D\in S_C}T_{C, \il D}$.
By definition we have 
\begin{alignat*}1
T_{C, \il D}&=\{\B\unlhd\O_\KL: \Lambda_C(\il D \il B)\neq \emptyset\}\\
&=\{\B\unlhd \O_\KL: \Lambda_C^*(\il E\il D \il B)\neq \emptyset \text{ for some }\il E\unlhd \O_\KL\}\\
&\subseteq\{\B\unlhd \O_\KL:\il E\il D \il B\in S_C\text{ for some }\il E\unlhd \O_\KL\}.
\end{alignat*}
Now using that 
$S_C=\{\up C,\sqrt{p}\up C\}$ and that $\sqrt{p}\O_\KL$ is a prime ideal we see that 
$T_{C, \il D}\subseteq \{\O_\KL, \sqrt{p}\O_\KL\}$ for any $\il D\in S_C$.
Thus $|T|= |\cup_{C\in \mathcal{R}}\cup_{\il D\in S_C}T_{C, \il D}|\leq 2$.
\end{proof}

\begin{lemma}\label{latticeinclusion}
Let $\sigma$ be as in (\ref{sigd}).
We have 
\begin{alignat*}1
\Lambda(\il A \il D, CE)\subseteq \sigma(CE)\times\sigma(CE). 
\end{alignat*}
Moreover, if $\il D=\sqrt{p}\up C$ then we have
\begin{alignat*}1
\Lambda(\il A \il D, CE)\subseteq \sigma\left(CE p(CE,p\O_\ks)^{-1}\right)\times\sigma(CE).
\end{alignat*}
\end{lemma}
\begin{proof}
The first assertion is clear from the definition. 
For the second assertion 
we could use the last equality in the proof of Lemma \ref{Lambdadeterminant}, but we
prefer to give a direct argument here.
Note that $\sigma\bbalf\in \Lambda(\il A \il D)$
implies $\il D\mid \iN(\bbalf)=(\balpha_0\O_\KL,\sqrt{p}\balpha_1\O_\KL)$.
As $\il D=\sqrt{p}\up C$ we conclude $\sqrt{p}\O_\KL\mid \balpha_0\O_\KL$,
and thus $p\O_\ks\mid \balpha_0\O_\ks$. Therefore $\balpha_0\in CE \cap p\O_\ks$.
This proves the second assertion.
\end{proof}
Next we use a trick, simpler but reminiscent of those used in \cite[Section 6]{Widmerintpts}.   
To this end we introduce a linear automorphism $\Phi$ of determinant $1$ on $\left(\IR^r\times\IC^s\right)^2$ by
\begin{alignat}3
\label{Phi}
\Phi(\vz_0,\vz_1):=(p^{-1/4}\vz_0,p^{1/4}\vz_1).
\end{alignat}
\begin{lemma}\label{appliedminimaestimates}
Write $\Lambda:=\Lambda(\il A \il D, CE)$. If $\il D=\up C$ then we have
\begin{alignat*}1
\lambda_1(\Phi\Lambda)&\geq p^{-1/4}\Normk(CE)^{1/d},\\
\lambda_{d+1}(\Phi\Lambda)&\geq p^{1/4}\Normk(CE)^{1/d}. 
\end{alignat*}
If $\il D=\sqrt{p}\up C$ then we have
\begin{alignat*}1
\lambda_1(\Phi\Lambda)&\geq \begin{cases}p^{-1/4}\Normk(CE)^{1/d} &\text{ if }p\O_\ks\mid E,  \\
                                        p^{1/4}\Normk(CE)^{1/d} &\text{ if }p\O_\ks\nmid E. 
                           \end{cases}\\
\lambda_{d+1}(\Phi\Lambda)&\geq\begin{cases}p^{1/4}\Normk(CE)^{1/d} &\text{ if }p\O_\ks\mid E,  \\
                                        p^{3/4}\Normk(CE)^{1/d} &\text{ if }p\O_\ks\nmid E.
                           \end{cases}
\end{alignat*}
\end{lemma}
\begin{proof}
By Lemma \ref{latticeinclusion} we have $\Phi\Lambda\subseteq \Lambda_1\times \Lambda_2$, where $\Lambda_2:=p^{1/4}\sigma(CE)$
and $\Lambda_1$ is $p^{-1/4}\sigma(CE)$ if $\il D=\up C$ and $p^{-1/4}\sigma\left(CE p(CE,p\O_\ks)^{-1}\right)$ if  $\il D=\sqrt{p}\up C$. 
Recall the fact (already used in Lemma \ref{appliedlatticeestimate})
that $\lambda_1(\sigma A)\geq \Normk A^{1/d}$ for any nonzero ideal $A$ of $\ks$. Using this and applying Lemma \ref{latticeminimaestimates}
the result follows from an easy computation.
\end{proof}

\begin{lemma}\label{LipBall}
There exist constants $c_1=c_1(\ks)$ and $M=M(\ks)$ depending solely on $\ks$ such that, with $L=c_1p^{-1/4}\TE$,
we have $\Phi S_F(\TE)\subseteq B_0(L)$ and the boundary $\partial\Phi S_F(\TE)\in \Lip(2d,M,L)$. 
\end{lemma}
\begin{proof}
The adelic Lipschitz system $\en$ on $\KL$ leads to an adelic Lipschitz system $\en'$ on $\ks$ as in Section \ref{section6}. The latter is used to define $S_F(\TE)$. 

Now notice that applying $\Phi$ to  $S_F(\TE)$ gives the same as defining $S_F(\TE)$ using the
standard adelic Lipschitz system defined by $N_v(z_0,z_1)=\max\{|z_0|_v,|z_1|_v\}$ for all $v$ and then homogeneously
shrinking this set by the factor $p^{-1/4}$. The claims then follow immediately from Lemma \ref{SFLip}, (\ref{kappa}),
and (\ref{SFnormbound}) applied to the standard adelic Lipschitz system. 
\end{proof}

\begin{lemma}\label{mainlemma2}
Let $\Error_1:=X^d/\Normk(E)$, and let $\Error_2:=X^{2d-1}/(p^{(d-1)/2}\Normk(E)^{2-1/d})$. Then we have 
\begin{alignat*}1
|\Lambda(\il A \il D, CE) \cap S_F(X \NormK\il D^{1/(2d)})|=
&\frac{\Vol S_F(1)\NormK\il DX^{2d}}{\det \Lambda(\il A \il D, CE)}\\
+&O\left(\begin{cases}
          \Error_1+\Error_2 &\text{ if }p\O_\ks\nmid E\\
          p^{d/2}\Error_1+p^{d-1/2}\Error_2 &\text{ if }p\O_\ks\mid E
         \end{cases}
\right).
\end{alignat*}
Moreover, there is a constant $\ga=\ga(\ks)\geq 1$ depending only on $\ks$, such that 
$|\Lambda(\il A \il D, CE) \cap S_F(X \NormK\il D^{1/(2d)})|=0$ whenever $\Normk E>(\ga p X)^d$.
\end{lemma}
\begin{proof}
First note that
\begin{equation*}
|\Lambda(\il A \il D, CE) \cap S_F(X \NormK\il D^{1/(2d)})|=|\Phi\Lambda(\il A \il D, CE) \cap \Phi S_F(X \NormK\il D^{1/(2d)})|.
\end{equation*}
Now we apply Lemma \ref{Lemmacountinglatticepts} with $a=d+1$ combined with Lemma \ref{LipBall} to conclude
\begin{alignat*}1
|\Phi\Lambda(\il A \il D, CE) \cap \Phi S_F(X \NormK\il D^{1/(2d)})|=
\frac{\Vol S_F(1)\NormK\il DX^{2d}}{\det \Lambda(\il A \il D, CE)}\\
+O\left(\max\left\{\frac{p^{-d/4}X^d \NormK\il D^{1/2}}{\lambda_1(\Phi\Lambda)^d},\frac{p^{-(2d-1)/4}X^{2d-1} \NormK\il D^{1-1/(2d)}}
{\lambda_1(\Phi\Lambda)^d\lambda_{d+1}(\Phi\Lambda)^{d-1}}\right\}\right).
\end{alignat*}
Finally, we use Lemma \ref{appliedminimaestimates} to estimate $\lambda_1(\Phi\Lambda)$ and $\lambda_{d+1}(\Phi\Lambda)$, and the first claim follows from
a simple computation. The second claim follows from Lemma \ref{appliedlatticeestimate} combined with (\ref{Cen'here}) and (\ref{nbp}).
\end{proof}
We are now in the position to prove Proposition \ref{properrorterm}. In the introduction we already computed the main term, see \eqref{example2}. Proceeding exactly as in the proof of Theorem \ref{generalthm} in the case $(n,d)=(1,1)$, we obtain  
\begin{alignat*}1
N_\en(\P^1(k);X)&=\frac{2p^{d/2}}{p^d+1}S_\ks(1)X^{2d}\\
&+O\left(\sum_{C \in \mathcal{R}}\sum_{\il D \in S_C}
\sum_{\il A \in T}\sum_{E \unlhd \O_\ks\atop \Normk E>(\ga p X)^d}\frac{\Vol\Phi  S_F(X \NormK\il D^{1/(2d)})}{\det\Phi \Lambda(\il A \il D, CE)}\right) \\
&+O\left(\sum_{C \in \mathcal{R}}\sum_{\il D \in S_C}
\sum_{\il A \in T}\sum_{E \unlhd \O_\ks\atop \Normk E\leq (\ga p X)^d}\Error_1+\Error_2\right)\\
&+O\left(\sum_{C \in \mathcal{R}}\sum_{\il D \in S_C}
\sum_{\il A \in T}\sum_{\substack{E \unlhd \O_\ks \\ \Normk E\leq (\ga p X)^d\\ p\O_\ks \mid E}}p^{d/2}\Error_1+p^{d-1/2}\Error_2 \right).
\end{alignat*}
For the first error term we apply Minkowski's second theorem and Lemma \ref{LipBall} to get the upper bound
\begin{alignat*}1
\frac{\Vol\Phi  S_F(X \NormK\il D^{1/(2d)})}{\det\Phi \Lambda(\il A \il D, CE)}\ll\frac{L^{2d}}{\lambda_1(\Phi\Lambda)^d\lambda_{d+1}(\Phi\Lambda)^{d}},
\end{alignat*}
where $L\ll p^{-1/4}X\NormK\il D^{1/(2d)}$. Summing the above over the finite sums can be handled by Lemmata \ref{sctwoelements} and  \ref{setcardbounds}. Now for the infinite sum over the ideals $E$, we apply Lemma \ref{appliedminimaestimates}, and a straightforward computation (using the dichotomy $P\mid E$, $P \nmid E$) yields the upper bound
\begin{alignat*}1
\ll \frac{X^{d}}{p^{3d/2}}.
\end{alignat*}
For the second error term we note that 
\begin{alignat*}1
\sum_{E \unlhd \O_\ks\atop \Normk E\leq (\ga p X)^d}\Error_1=\sum_{E \unlhd \O_\ks\atop \Normk E\leq (\ga p X)^d}\frac{X^d}{\Normk E}\ll X^d\log((\ga p X)^d)\ll X^d\log X+X^d\log p,
\end{alignat*}
and
\begin{alignat*}1
\sum_{E \unlhd \O_\ks\atop \Normk E\leq (\ga p X)^d}\Error_2\leq
\frac{X^{2d-1}}{p^{(d-1)/2}}\sum_{E \unlhd \O_\ks}\Normk E^{-2+1/d}\ll\frac{X^{2d-1}}{p^{(d-1)/2}}.
\end{alignat*}
Then we apply Lemmata \ref{sctwoelements} and  \ref{setcardbounds} to conclude
\begin{alignat*}1
\sum_{C \in \mathcal{R}}\sum_{\il D \in S_C}
\sum_{\il A \in T}\sum_{E \unlhd \O_\ks\atop \Normk E\leq (\ga p X)^d}\Error_1+\Error_2 \ll X^d\log X+X^d\log p+\frac{X^{2d-1}}{p^{(d-1)/2}}.
\end{alignat*}
Similar straightforward calculations yield
\begin{alignat*}1
\sum_{\substack{E \unlhd \O_\ks \\ \Normk E\leq (\ga p X)^d\\ p\O_\ks \mid E}}p^{d/2}\Error_1\ll \frac{X^{d}}{p^{d/2}}\log X,
\end{alignat*}
and 
\begin{alignat*}1
\sum_{\substack{E \unlhd \O_\ks \\ \Normk E\leq (\ga p X)^d\\ p\O_\ks \mid E}}p^{d-1/2}\Error_2\ll \frac{X^{2d-1}}{p^{3d/2-1}}.
\end{alignat*}
Thus, applying again Lemmata \ref{sctwoelements} and  \ref{setcardbounds}, we see that 
\begin{alignat*}1
\sum_{C \in \mathcal{R}}\sum_{\il D \in S_C}
\sum_{\il A \in T}\sum_{\substack{E \unlhd \O_\ks \\ \Normk E\leq (\ga p X)^d\\ p\O_\ks \mid E}}p^{d/2}\Error_1+p^{d-1/2}\Error_2
\ll X^d\log X+\frac{X^{2d-1}}{p^{(d-1)/2}}.
\end{alignat*}
Combining these estimates and Lemma \ref{onetoone} completes the proof of Proposition \ref{properrorterm}.

We can now sum $N(\sqrt{p}\ks^*,X)$ over all $p\in \mathbf{P}_\ks$. The next lemma tells us that we can restrict the summation to $p\leq X^2$.
\begin{lemma}\label{pX2}
For any $\alpha\in \ks^*$ and any $p\in \mathbf{P}_\ks$ we have $H(\sqrt{p}\alpha)\geq \sqrt{p}$. 
\end{lemma}
\begin{proof}
Let $x\in \KL$ and let $\mathfrak{P}$ be the prime ideal $\sqrt{p}\O_\KL$. Then 
\begin{alignat*}1
H(x)\geq \max\{1,\NormK \mathfrak{P}\}^{-v_\mathfrak{P}(x\O_\KL)/(2d)}=\max\{1,p^d\}^{-v_\mathfrak{P}(x\O_\KL)/(2d)}.
\end{alignat*}
In particular, if $v_\mathfrak{P}(x\O_\KL)<0$ we get $H(x)\geq \sqrt{p}$. 
As $H(x)=H(1/x)$ for any nonzero $x$ whatsoever, it suffices to show that the order of $\sqrt{p}\alpha\O_\KL$ at $\mathfrak{P}$ is nonzero.
As $p$ is inert in $\ks$ the order of $\alpha \O_\KL$ at $\mathfrak{P}$ is even. Hence 
the order of $\sqrt{p}\alpha\O_\KL$ at $\mathfrak{P}$ is odd. 
\end{proof}
We can now prove Theorem \ref{Thmvartheta}. 
Clearly, we have 
\begin{alignat*}1
N(\sqrt{\mathbf{P}_\ks} \ks,X)&=1+\sum_{p\in \mathbf{P}_\ks\atop p\leq X^2}N(\sqrt{p}\ks^*,X)\\
=&\sum_{p\in \mathbf{P}_\ks\atop p\leq X^2}\frac{2p^{d/2}}{p^d+1}S_\ks(1)X^{2d}+O\left(\frac{X^{2d-1}}{p^{(d-1)/2}}+X^d\log X+X^d\log p\right) \\
=&\sum_{p\in \mathbf{P}_\ks \atop p\leq X^2}\frac{2p^{d/2}}{p^d+1}S_\ks(1)X^{2d}
+O\left(\sum_{p\in \mathbf{P}_\ks\atop p\leq X^2}\frac{X^{2d-1}}{p^{(d-1)/2}}\right)+O\left(\sum_{p\in \mathbf{P}_\ks\atop p\leq X^2}X^d\log X\right). 
\end{alignat*}
By the prime number theorem we have
\begin{alignat*}1
\sum_{p\in \mathbf{P}_\ks\atop p\leq X^2}X^d\log X\ll X^{d+2}.
\end{alignat*}
A straightforward calculation yields
\begin{alignat*}1
\sum_{p\in \mathbf{P}_\ks\atop p\leq X^2}\frac{X^{2d-1}}{p^{(d-1)/2}}\ll \begin{cases}
                                                                          X^{2d-1} &\text{ if }d\geq 4,\\
	                                                                  X^{5}\log\log X &\text{ if }d=3,\\
									  X^{4} &\text{ if }d=2.
                                                                         \end{cases}
\end{alignat*}
To handle the first term let us start with the simpler case $d\geq 3$. Then we have
\begin{alignat*}1
\sum_{p\in \mathbf{P}_\ks\atop p\leq X^2}\frac{2p^{d/2}}{p^d+1}S_\ks(1)X^{2d}
&=\sum_{p\in \mathbf{P}_\ks}\frac{2p^{d/2}}{p^d+1}S_\ks(1)X^{2d}+O\left(\sum_{p\in \mathbf{P}_\ks\atop p>X^2}\frac{2p^{d/2}}{p^d+1}S_\ks(1)X^{2d}\right)\\
&=\sum_{p\in \mathbf{P}_\ks}\frac{2p^{d/2}}{p^d+1}S_\ks(1)X^{2d}+O(X^{2d-1}).
\end{alignat*}
This finishes the proof of Theorem \ref{Thmvartheta} for $d\geq 3$.

Let us now assume $d=2$. It remains to show that 
\begin{alignat*}1
\sum_{p\in \mathbf{P}_\ks\atop p\leq X^2}\frac{2p}{p^2+1}=\log\log X +O(1).
\end{alignat*}
Clearly, we have 
\begin{alignat*}1
\sum_{p\in \mathbf{P}_\ks\atop p\leq X^2}\frac{2p}{p^2+1}=
\sum_{p\in \mathbf{P}_\ks\atop p\leq X^2}\frac{2}{p}+O(1).
\end{alignat*}
By an explicit version of Chebotarev's density theorem (see, e.g., \cite{LagariasOdlyzko}) we know that for $T\geq 3$ (using $\Li(T) = T/\log{T} + O(T/(\log T)^2)$)
\begin{alignat*}1
\sum_{p\in \mathbf{P}_\ks\atop p\leq T}1=\frac{T}{2\log T}+ O\left(\frac{T}{(\log T)^2}\right).
\end{alignat*}
Applying partial summation we get
\begin{alignat*}1
\sum_{p\in \mathbf{P}_\ks\atop p\leq X^2}\frac{2}{p}=\sum_{m=2}^{X^2}\frac{1}{(m+1)\log m}+O(1)=\log\log X +O(1).
\end{alignat*}
This completes the proof  of Theorem \ref{Thmvartheta} for $d=2$.

\section*{Appendix}
\setcounter{alemma}{0}
    \renewcommand{\thealemma}{A\arabic{alemma}}
We will now apply Theorem \ref{generalthm} to deduce the formula (\ref{lineqexample}).
We start by proving our claim that $\en'$ is an adelic Lipschitz system whenever all the functions $N_w$ of $\en$ are norms. 
To this end we shall use the following simple observations.\\

Let  $f_1,f_2,f: \IR^q\rightarrow \IR$ and $F:[0,1]^{q-1}\rightarrow \IR^q$ be functions that satisfy a Lipschitz condition with
Lipschitz constant $L_{f_1}, L_{f_2}, L_{f}$ and $L_{F}$ respectively. Then we have:
\begin{enumerate}[1.]
\item $|f(F(\vt))-f(F(\vt'))|\leq L_fL_F|\vt-\vt'|$ for all $\vt, \vt'\in [0,1]^{q-1}$.
\item Suppose that $f(F(\vt))\geq c>0$ for all $\vt \in
  [0,1]^{q-1}$ and let $\alpha\leq 1$. Then
  $|f(F(\vt))^\alpha-f(F(\vt'))^\alpha|\leq |\alpha|c^{\alpha-1}L_fL_F |\vt-\vt'|$ for all $\vt, \vt' \in [0,1]^{q-1}$. (We use the convention that $0^0=1$.)
\item Suppose that $|f_1(F(\vt))|,|f_2(F(\vt))|,|f(F(\vt))|,|F(\vt)|\leq C$ for all $\vt \in [0,1]^{q-1}$. Then, for all $\vt, \vt' \in [0,1]^{q-1}$,
  \begin{enumerate}[(a)]
    \item $|f_1(F(\vt))f_2(F(\vt))-f_1(F(\vt'))f_2(F(\vt'))|\leq C(L_{f_1}+L_{f_2})L_F|\vt-\vt'|$,
    \item $|f(F(\vt))F(\vt)-f(F(\vt'))F(\vt')|\leq CL_F(L_f+1)|\vt-\vt'|$.
  \end{enumerate}
\end{enumerate}

Here 1.\ is obvious, 2.\ follows from the mean value theorem and 1., and  3.\ (a) and (b) are consequences of the identity $fg-f'g'=(f-f')g+f'(g-g')$ and 1.\ (note that the assumption $|F(\vt)|\leq C$ is needed only for (b)).

\begin{alemma}\label{en-norms}
Let $\en$ be an adelic Lipschitz system (of dimension $n$) on $\KL$ and assume that for every Archimedean place $w$ of $\KL$ the function $N_w$ satisfies a Lipschitz condition. 
Then $\en'=\en'(\en,\ks)$ is an adelic Lipschitz system (of dimension $n$) on $\ks$.  
\end{alemma}

\begin{proof}
  The conditions $(i), (ii)$ and $(iv)$ in Definition \ref{defALS} are
  obviously satisfied. It remains to prove $(iii)$. Given an
  Archimedean place $v$ of $\ks$, let
  $\rho:[0,1]^{d_v(n+1)-1}\rightarrow \IS^{d_v(n+1)-1}$ be the
  (normalized) standard parameterization via polar coordinates of the
  ${(d_v(n+1)-1)}$-dimensional unit sphere in $\ks_v^{n+1}$. Then $\rho$
  is Lipschitz. The subset of $\ks_v^{n+1}$ where $N_v(\vz)=1$ is
  parameterized by the function $\s : [0,1]^{d_v(n+1)-1} \to
  \ks_v^{n+1}$, defined by $\s(\vt) := 1/N_v(\rho(\vt))\cdot
  \rho(\vt)$. Let us show that $\s$ satisfies a Lipschitz condition. 
  
  For any Archimedean place $w$ of $\KL$ extending $v$, the function
  $N_w$ is continuous and nonzero on the compact set
  $\IS^{d_v(n+1)-1}$, whence $1 \ll_\en N_w(\rho(\vt)) \ll_\en 1$ on
  $[0,1]^{d_v(n+1)-1}$. Thus, $N_{w}(\rho(\vt))^{-\frac{d_{w}}{d_v[\KL
      : \ks]}}$ is bounded, and by 2.\ satisfies a Lipschitz
  condition.  Hence, by 3.\ (a) also $N_{v}(\rho(\vt))^{-1}$ is Lipschitz. By 3.\
  (b), we conclude that $\s$ satisfies a Lipschitz condition.
\end{proof}

Note that any norm $\|\cdot\|$ on $\IR^q$ satisfies a Lipschitz condition. This follows from the reverse triangle 
inequality $|\|x\|-\|y\||\leq \|x-y\|$ and the equivalence of all norms on $\IR^q$. Thus, 
if all the functions $N_w$ are norms then Lemma \ref{en-norms} applies and so $\en'=\en'(\en,\ks)$ is an adelic Lipschitz system (of dimension $n$) on $\ks$.
More generally, let $B_w := \{\vz \in \KL_w^{n+1} : N_w(\vz) \leq 1\}$ be the
compact star-shaped body corresponding to $N_w$.  Let $\ker(B_w)$ be the convex
kernel of $B_w$, that is the set of all $\vz \in B_w$ such that for all
$\vz' \in B_w$ the line segment $[\vz,\vz']$ is contained in $B_w$. Then $\vNull
\in \ker(B_w)$ and $B_w$ is convex if and only if $\ker(B_w) = B_w$. Moreover, 
\cite[Lemma 1]{Beer} tells
us that $N_w$ is Lipschitz whenever $\vNull$ is in the interior of
$\ker(B_w)$.

Let us now show how the formula (\ref{lineqexample}) follows from Theorem \ref{generalthm}.
We use the adelic Lipschitz system $\en$ (of dimension $2$) on $\KL := \Q(\sqrt{2}, \sqrt{3}, \sqrt{5})$ defined by
\begin{align*}
N_w(z_0, z_1, z_2) :&= \max\{|z_0|_w, |z_1|_w, |z_2|_w, |\frac{\sqrt{2}z_1 + \sqrt{3}z_2}{\sqrt{5}}|_w\}\text,
\end{align*}
for any place $w$ of $\KL$. 
Hence all the $N_w$ are norms so that, thanks to Lemma \ref{en-norms}, we can apply Theorem \ref{generalthm}.
With the notation from Section \ref{section6}, we have 
$\lineqcountingfunction(X) = N_\en(\P^2(\Q),X)+O(X^2)$, as already mentioned in the introduction. Here the error term accounts for the
projective points of the form $(0:\balpha_1:\balpha_2)$.
With Theorem \ref{generalthm}, the only remaining task is to calculate $g_\Q^\en$.

\begin{alemma}
We have
\[g_\Q^\en = \frac{1}{31 \zeta(3)}(1 + 2\cdot 5^{1/4} + 4\cdot 5^{-1/2})\text.\]
\end{alemma}

\begin{proof}
For some tedious computations in $K$, we use the computer algebra system Sage\footnote{\url{http://www.sagemath.org}}. We use the same notation as in Section \ref{section6}. Clearly, we can choose $\mathcal{R} = \{\Z\}$. For any $\bbalf = (\balpha_0, \balpha_1, \balpha_2) \in \Q^3$, we have
\[\iN(\bbalf) = \balpha_0\O_\KL + \balpha_1\O_\KL + \balpha_2\O_\KL + \frac{\sqrt{2}\balpha_1 + \sqrt{3}\balpha_2}{\sqrt{5}}\O_\KL\text.\]
If $\O_\Q(\bbalf) = \Z$ then $\balpha_0\O_\KL + \balpha_1\O_\KL + \balpha_2\O_\KL = \O_\KL$, so $\iN(\bbalf) \supseteq \O_\KL$. On the other hand, we 
clearly have $\iN(\bbalf) \subseteq (\sqrt{5})^{-1}O_\KL$. Thus, we can choose
\[S_\Z := \{(\sqrt{5})^{-1}\il D: \il D \mid \sqrt{5}\O_\KL\}\text.\]
Moreover, if $\bbalf \in \Lambda_\Z((\sqrt{5})^{-1}\il D \il A)$, for some nonzero ideal $\il A$ of $\O_\KL$, then $\iN(\bbalf) = (\sqrt{5})^{-1}\il D_1$, 
for some nonzero ideal $\il D_1 \mid \sqrt{5}\O_\KL$. In particular, $\il D\il A \mid \il D_1$. This shows that $T_{\Z, (\sqrt{5})^{-1}\il D}$ is contained in the finite set
\[T := \{\il A: \il A \mid\sqrt{5}\O_\KL\}\text.\]
With \eqref{general_formula}, we obtain
\begin{equation}\label{pf_lineq_general_formula}
g_\Q^\en = \sum_{\il D \mid \sqrt{5}\O_\KL}\NormK((\sqrt{5})^{-1}\il D)^{3/8}\sum_{\il A \mid \sqrt{5}\O_\KL}\mu_\KL(\il A)\Sigma(\il A \il D)\text,
\end{equation}
where 
\[\Sigma(\il B) := \sum_{n=1}^{\infty}\frac{\mu(n)}{\det\Lambda((\sqrt{5})^{-1}\il B, n\Z)}\text.\]
Let us evaluate this sum for any ideal $\il B$ of $\O_\KL$ dividing $5 \O_\KL$. Elementary manipulations show that $\Lambda((\sqrt{5})^{-1}\il B, n\Z)$ is the sublattice of $\Z^3$ consisting of all
\begin{equation}\label{pf_lineq_lattice}
\bbalf = (\balpha_0, \balpha_1, \balpha_2) \in (n\Z \cap (\sqrt{5})^{-1} \il B)^3 \text{ such that } \sqrt{2}\balpha_1 + \sqrt{3}\balpha_2 \in \il B\text. 
\end{equation}
We have $5 \O_\KL = \il P_1^2 \il P_2^2$, where
\begin{equation*}
  \il P_1 := (5,
  \sqrt{15}-\sqrt{10}+\sqrt{6}-1)\text{, }\il
  P_2 := (5, \sqrt{15}-\sqrt{10}+\sqrt{6}+1 ) 
\end{equation*}
are distinct prime ideals of $\O_\KL$ with inertia degrees equal to
$2$.

For $\il B = \O_K$, the first condition in \eqref{pf_lineq_lattice} amounts to $\bbalf \in (n\Z)^3$. Then the second condition is always satisfied, and $\det \Lambda((\sqrt{5})^{-1}\O_\KL, n\Z) = n^3$. Therefore,
\begin{equation}\label{pf_lineq_sum_Ok}
\Sigma(\O_\KL) = \sum_{n=1}^{\infty}\frac{\mu(n)}{n^3} = \frac{1}{\zeta(3)}\text.
\end{equation}

If $\il B = \il P_1$, then the first condition in \eqref{pf_lineq_lattice} is equivalent to $\bbalf \in (n\Z)^3$. For the second condition, we find that $-(\sqrt{3})^{-1}\sqrt{2} \equiv 3 \mod \il P_1$,
so this condition is equivalent to $\balpha_2 = 3 \balpha_1 + a$, for an $a \in \il P_1 \cap n\Z = \lcm(5, n)\Z$. Therefore, $\Lambda((\sqrt{5})^{-1}\il P_1, n\Z)$ has the basis 
\[\{(n,0,0), (0,n,3n), (0,0,\lcm(5,n))\}\]
of determinant $n^2 \lcm(5, n)$. A similar computation shows that $-(\sqrt{3})^{-1}\sqrt{2} \equiv 2 \mod \il P_2$, so
\[\{(n,0,0), (0,n,2n), (0,0,\lcm(5,n))\}\]
is a basis of $\Lambda((\sqrt{5})^{-1}\il P_2, n\Z)$ of the same determinant. Thus,
\begin{equation}\label{pf_lineq_sum_Pi}
\Sigma(\il P_i) = \sum_{n=1}^{\infty}\frac{\mu(n)}{n^2\lcm(5, n)} = \frac{1}{\zeta(3)}\frac{5^2-1}{5^3-1}\text.
\end{equation}

For $\il B = \il P_1 \il P_2 = \sqrt{5} \O_K$, the first condition in \eqref{pf_lineq_lattice} is again equivalent to $\bbalf \in (n\Z)^3$. The second condition is equivalent to $\balpha_2 \equiv -(\sqrt{3})^{-1}\sqrt{2} \balpha_1 \mod \il P_1 \il P_2$. By the Chinese remainder theorem and what we have seen before, this is equivalent to
\[\balpha_2 \equiv 2 \balpha_1 \mod 5 \quad\text{ and }\quad \balpha_2 \equiv 3 \balpha_1 \mod 5\text,\]
so $\balpha_1 \equiv \balpha_2 \equiv 0 \mod 5$. Thus, $\Lambda((\sqrt{5})^{-1}\il P_1\il P_2, n\Z) = n\Z \times (\lcm(5, n) \Z)^2$ has determinant $n \lcm(5, n)^2$. We obtain
\begin{equation}\label{pf_lineq_sum_P1P2}
\Sigma(\il P_1 \il P_2) = \sum_{n=1}^{\infty}\frac{\mu(n)}{n\lcm(5, n)^2} = \frac{1}{\zeta(3)}\frac{5-1}{5^3-1}\text.
\end{equation}

In the other cases, that is $\il P_1^2 \mid \il B$ or $\il P_2^2 \mid \il B$, we have $\down((\sqrt{5})^{-1}\il B) = 5\Z$, so the first condition in \eqref{pf_lineq_lattice} is equivalent to $\bbalf \in (\lcm(5, n)\Z)^3$. In this case, the second condition is always satisfied, so we obtain $\det\Lambda((\sqrt{5})^{-1}\il B, n\Z) = \lcm(5, n)^3$ and
\begin{equation}\label{pf_lineq_sum_other_B}
\Sigma(\il B) = \sum_{n=1}^{\infty}\frac{\mu(n)}{\lcm(5, n)^3} = 0\text.
\end{equation}
A simple computation shows that 
\[\NormK((\sqrt{5})^{-1}\O_\KL)^{3/8} = 5^{-3/2}\text{, }\quad \NormK((\sqrt{5})^{-1}\il P_i)^{3/8} = 5^{-3/4}\text{, }\quad \NormK(\O_\KL)^{3/8} = 1\text.\]
To prove the lemma, just substitute this and \eqref{pf_lineq_sum_Ok} -- \eqref{pf_lineq_sum_other_B} in \eqref{pf_lineq_general_formula}.
\end{proof}

\section*{Acknowledgments}
We would like to thank David Masser for having brought our attention
to some of the problems considered here, and Robert Tichy for giving
us the opportunity to start this collaboration during our common time
in Graz. This work was completed while the second author was a
Visiting Fellow of the Center for Advanced Studies at LMU M\"unchen in
August 2012. He would like to thank Ulrich Derenthal for the
invitation, and the CAS for the financial support.  Finally, we are
indebted to the referee for the very careful reading and an excellent
report with many valuable and detailed suggestions.  The referee also
alerted us to an error in the statement of Lemma \ref{en-norms}, which
is now corrected.

\bibliographystyle{alpha}
\bibliography{literature}

\end{document}